\documentclass[11pt,letterpaper]{amsart}

\usepackage{header}
\usepackage{url}

\begin{document}
\title{Directional Ballistic transport for partially periodic Schr\"{o}dinger operators on $\mathbb Z^2$}
\author{Adam Black}
\address{Department of Mathematics\\
  University of California, Berkeley\\
 Berkeley, CA 94720}
\email[]{adamblack@berkeley.edu}
\thanks{This material is based upon work supported by the National Science Foundation under Awards No. 2503339}
\author{David Damanik}
\address{Department of Mathematics\\
  Rice University\\
 Houston, TX, 77005}
\email[]{damanik@rice.edu }
\thanks{D.\ D.\ was supported in part by NSF grants DMS--2054752 and DMS--2349919}
\author{Tal Malinovitch}
\address{Department of Mathematics\\
  Rice University\\
 Houston, TX 77005}
  \email[]{tal.malinovitch@rice.edu}
\author[Young]{Giorgio Young}
\address{Department of Mathematics\\
  The University of Wisconsin\\
 Madison, WI 53706}
  \email[]{gfyoung@wisc.edu}
  \thanks{G.Y.\ acknowledges the support of the National Science Foundation through grant DMS--2303363.}
\maketitle
\begin{abstract}
     We study the transport properties of Schr\"{o}dinger operators on $\bbZ^2$ with potentials that are periodic in one direction and compactly supported in the other. Such systems are known to produce \emph{surface states} that are weakly confined near the support of the potential. We demonstrate that surface states exhibit what we describe as directional ballistic transport, characterized by a strong form of ballistic transport along the periodic direction and its absence in the compactly supported one. By showing that the scattering states exhibit ballistic transport, we obtain ballistic transport for a dense subset of all of $\ell^2(\mathbb {Z} ^2)$.
\end{abstract}

\section{Introduction}
\subsection{Motivation and main results}
Recently, there has been much interest in Schr\"{o}dinger operators 
\begin{align*}
     H=-\frac{1}{2}\Delta+V   
\end{align*}
 with potentials $V$ supported near a hyperplane, see, e.g., \cite{bentosela2012guided,black2022scattering, filonov, filonovKlopp, hoang2014absence, Richard} and the references therein.
 Intuitively, one expects that any solution to the associated Schr\"{o}dinger equation, given by
\begin{align}\label{eq:Schrodt}
    i\partial_t\psi=H\psi, \; \psi(0)=\psi_0,
\end{align}
should decompose into a piece that radiates into the background and one that is governed by the lower-dimensional dynamics.
Indeed, it was shown in \cite{black2022scattering} that if $V$ is a real-valued bounded potential supported near a proper subspace of $\bbR^d$, then $L^2(\bbR^d)$ decomposes into the space of scattering states and the space of \emph{surface states} that may only evolve away from the potential at a sublinear rate. This result is agnostic to the exact choice of the potential $V$, so it is natural to ask how different properties of $V$ affect the dynamics within the surface subspace. 

In this paper, we study partially periodic surface potentials on $\bbZ^2$, and in particular their dynamical properties. Among various choices of surface potentials, partially periodic ones are important because they may serve as waveguides; see, e.g., \cite{photonicsreview,johnson2000linear,korotyaev2017schrodinger,meade2008photonic} for related systems in the photonics literature. In this context, this means that the surface states should exhibit transport primarily along the surface while being essentially trapped in the transverse directions.
Although such a picture is physically quite natural, little seems to be known rigorously about the dynamics of surface states.
To our knowledge, the only work addressing the dynamics of surface states in partially periodic models is the work of Davies-Simon \cite{DaviesSimon} who showed that surface states are concentrated near the support of the potential in a rather weak sense, see \eqref{eq:DSdynamical} below. In particular, this result is silent on the nature of the propagation in the directions of periodicity.  
More generally, partially periodic systems remain poorly understood compared to their fully periodic counterparts. We refer the reader to Section 9.4 of \cite{Kuchment} for an overview of what is known mathematically.
Recent works on such models \cite{filonovKlopp,filonov,hoang2014absence}, reviewed below, have sought to establish continuous spectrum for various classes of continuum partially periodic surface models.
Using such results, one may appeal to the RAGE theorem \cite{AmreinGeorgescu, Ruelle} and conclude that solutions to \eqref{eq:Schrodt} spread in a qualitative sense. However, this conclusion alone is unsatisfactory on two counts;  the first is that it does not reflect the anisotropy of the model, in particular, one may wonder about the direction of spread for surface states, and the second is that it does not reflect the more quantitative notions of spread we may expect from the periodicity of the embedded lower-dimensional model.
Indeed, it is classical that fully periodic potentials are known to induce a particularly strong form of ballistic transport \cite{AschKnauf}, so one expects an analogous \emph{directional} result for surface states in the partially periodic setting, in addition to an equivalently strong ballistic transport result for all states. Thus, our work here seeks to address both of these dynamical questions.

To be more precise, let us now specify the class of operators we wish to consider. Throughout, let the first coordinate of $\bbZ^2$ be labeled by $x$ and the second by $y$. We say that a real-valued potential $V(x,y)$ in $\ell^\infty(\bbZ^{2})$ is \emph{strip periodic} if $V$ is supported in $\{-R,-R+1,\ldots,R\}\times\bbZ$ for some $R$ in  $\bbZ_{>0}$ and there exists a period $L$ in $\bbZ_{>0}$ such that 
\begin{align*}
        V(x,y+L)=V(x,y),\, \,\,\forall (x,y)\in \bbZ^{2}.
\end{align*}
Furthermore, let $H_0$ be the Laplacian on $\bbZ^2$, i.e., 
\begin{align*}
    (H_0\psi)(x,y)=\psi(x+1,y)+\psi(x-1,y)+\psi(x,y+1)+\psi(x,y-1).
\end{align*}
For $V$ strip periodic on $\ell^2(\bbZ^{2})$, we consider the self-adjoint Schr\"{o}dinger operator
\begin{align}\label{Operator}
    H=H_0+V.
\end{align}

To state our first dynamical result, we introduce the position operators $X$ and $Y$
\begin{align*}
    (X\psi)(x,y)=x\psi(x,y),\;(Y\psi)(x,y)=y\psi(x,y),
\end{align*}
with the natural domains $D(X)$ and $D(Y)$.
We also introduce the vector-valued operator $\vec{Q}=(X,Y)$ with domain $D(\vec{Q})=D(X)\cap D(Y)$.
For an operator $A$, we denote by $A_H(t)$ the Heisenberg-evolved operator $e^{itH}Ae^{-itH}$. We say that a state $\psi\in  D(\vec{Q})$ undergoes \emph{ballistic transport} if 
\begin{align*}
    \lim_{t\rightarrow\infty}\frac{1}{t}\vec{Q}_H(t)\psi
\end{align*}
exists and is non-zero. As we will discuss below, this notion of ballistic transport is a particularly strong one. The fact that wave packets in crystals undergo ballistic transport is related to the motion of electrons in solids and has been established mathematically, first in continuum models in the 1990's \cite{AschKnauf}, later for discrete models in one and arbitrary dimensions, respectively \cite{damanik2015quantum, Fillman2021}, and quite recently extended to more general graphs \cite{MS}. It will follow from our main theorem that a dense set of states exhibits ballistic transport in this form:

\begin{theorem}\label{mainCor}
        Let $V$ be a strip periodic potential on $\bbZ^{2}$ and let $H$ be the associated Schr\"{o}dinger operator $H=H_0+V$. Then a dense subset of $\ell^2(\bbZ^{2})$ exhibits ballistic transport.
\end{theorem}

Our main theorem is concerned with the subspace of surface states, which we define via the partial Floquet transform $\calU$. The map $\calU$, defined in Section~\ref{existenceSection}, is a unitary map that conjugates $H$ to a direct integral of operators  
\begin{align*}
   \int\limits_{\calB}^\oplus H(k)\:\frac{dk}{\abs{\calB}},
\end{align*}
where $\calB=[0,2\pi/L)$ is the Brillouin zone and each $H(k)$ is a self-adjoint operator on the discrete cylinder $\bbZ\times \bbZ_L$.
Relative to this decomposition, the space of surface states is given by
\begin{align*}
    \calH_{\textrm{sur}}:=\calU^*\int\limits_{\calB}^\oplus \mathcal{H}_{\textrm{pp}}(H(k))\:\frac{dk}{\abs{\calB}}.
\end{align*}
Note that by Appendix A of~\cite{DaviesSimon}, the projection onto $\calH_{\textrm{pp}}(H(k))$ is a measurable function of $k$, so this direct integral is well-defined.
We also mention that this subspace exactly coincides with the surface subspace defined for general strip potentials in \cite{black2022scattering}; see Section 6.2 of that work.

From the work of \cite{AschKnauf}, it is well-known that the transport properties of fully periodic operators are a consequence of the analytic variation of Bloch waves. Thus, our next theorem provides a suitable analytic Bloch theory in the partially periodic context:
\begin{theorem}
\label{thm:analyticity}
    Let $H$ be defined as in \eqref{Operator}. Then there is a countable collection of open sets $U_l\subset\calB $, countably many functions $\lambda_{l,i}:U_l\rightarrow \bbR$, and corresponding finite rank projectors $k\in \calB\mapsto \pi_{l,i}(k):\ell^2(\bbZ\times \bbZ_L)\rightarrow \ell^2(\bbZ\times\bbZ_L)$ with pairwise disjoint ranges such that
    \begin{enumerate}
        \item\label{item:thm1} We may write
        \begin{align*}
            \calU H\vert_{\calH_{\mathrm{sur}}}\calU^*=\sum_{l,i} \int\limits_{U_l}^\oplus\lambda_{l,i}(k)\pi_{l,i}(k)\:\frac{dk}{|\calB|}.
        \end{align*}
        \item\label{item:thm2}  Each $\lambda_{l,i}$ and $\pi_{l,i}$ is analytic on the complement of a finite set.
        \item\label{item:thm3} The set of $k$ for which $\frac{d\lambda_{l,i}}{dk}=0$ is measure zero. 
    \end{enumerate}
\end{theorem}
We note that the proof of this theorem is more involved than the analogous fully periodic result \cite{wilcox1978theory} because the operators $H(k)$ do not have compact resolvent.
This creates the possibility of eigenvalues embedded in the essential spectrum of $H(k)$ whose $k$-variation requires some care to analyze.
See Section~\ref{overviewSection} for a more in-depth discussion.

This theorem, combined with the characterization of $(\mathcal{H}_{\textrm{sur}})^\perp$ as the space of scattering states (see Proposition~\ref{Completeness}), then yields an immediate corollary on the spectral type of $H$:
\begin{corollary}
    The operator \eqref{Operator} has purely absolutely continuous spectrum.
\end{corollary}

This corollary forms a discrete analog of the main result in \cite{filonovKlopp} for the $\bbZ^2$ setting, with a different method of proof. 

In the partially periodic setting, one expects a more refined notion of ballistic transport that accounts for the anisotropy of the system to hold for states in $\mathcal H_{\textrm{sur}}$: we say that $\psi\in D(\vec{Q})$ exhibits \emph{directional ballistic transport} if 
\begin{align*}
    \lim_{t\rightarrow\infty}\frac{1}{t}X_H(t)\psi=0
\end{align*}
and
\begin{align*}
    \lim_{t\rightarrow\infty}\frac{1}{t}Y_H(t)\psi
\end{align*}
exists and is non-zero. In other words, there is a nonzero asymptotic velocity in the periodic direction, and an asymptotic velocity of zero in the compactly supported direction. This latter condition is referred to as an absence of ballistic transport.
Our main theorem is then as follows:
\begin{theorem}\label{thm:Z2ballistic}
    Let $V$ be a strip periodic potential on $\bbZ^{2}$ and $H$ the associated Schr\"{o}dinger operator $H=H_0+V$. Then any $\psi\in D(\vec{Q})\cap \calH_{\mathrm{sur}}\setminus\{0\}$ exhibits directional ballistic transport.
\end{theorem}

Before turning to a discussion of prior work, we also mention two auxiliary results that may be of independent interest. First, the proof of the absence of ballistic transport in the $x$-direction is closely related to Simon's classic result on the absence of ballistic transport for pure point states \cite{simon1990absence}. To our knowledge, the forms of this theorem that appear in the literature pertain to operators with only pure point spectrum. In contrast, for our purposes, we must generalize to operators that may also have continuous spectrum, (see Theorem~\ref{Thm:Simon}). This result is natural, in particular, because many important operators have both spectral types, especially on $\bbR^d$. In fact, the existence of bounded potentials in the continuum setting that induce completely pure point spectrum is a hard problem. In one dimension, this was established for the continuum Anderson model with absolutely continuous random variables by Kotani-Simon \cite{kotaniSimon} and for singular random variables by Damanik-Sims-Stolz \cite{DSS}. In higher dimensions, this remains a major open problem.\par
Second, Theorem \ref{mainCor} is a consequence of the fact that scattering states exhibit ballistic transport in the anisotropic setting. While this fact is known for short-range potentials \cite{damanik2025ballistic}, we are unaware of a reference that applies to potentials with anisotropic short-range decay. So, we provide a general criterion for a scattering state to exhibit ballistic transport \emph{in a certain direction}. This is inspired by Cook's criterion for the existence of the wave operator.

\subsection{Prior work}
In both the discrete and continuum settings, related classes of operators have been studied since at least the work of Davies-Simon \cite{DaviesSimon}, who studied potentials on $\bbR^d$ that decay in the $x_1$ direction, but are periodic in all the others. 
Their Theorem~6.1 shows that any surface state $\psi$ obeys
\begin{align}
\label{eq:DSdynamical}
    \lim_{a\rightarrow\infty}\sup_t\int_{|x_1|>a}|(e^{-itH}\psi)(x)|^2\:dx=0.
\end{align}
More recently, in the continuum setting, Filonov-Klopp \cite{filonovKlopp} and Filonov \cite{filonov} have shown that if $V$ is periodic in some variables, under different decay assumptions in the other variables, the operator $H$ has either no eigenvalues or purely absolutely continuous spectrum. Specifically, if $\abs{V(x,y)}<C|x|^{-\rho}$ for $\rho>1$, then $H$ has no eigenvalues \cite{filonov} and if $V$  decays superexponentially, then $H$ has purely absolutely continuous spectrum \cite{filonovKlopp}. Prior to Filonov \cite{filonov}, Hoang-Radosz \cite{hoang2014absence} obtained a similar result on the absence of eigenvalues for Helmholtz and Schr\"{o}dinger operators on $\bbR^2$. Notably, the latter authors were able to treat periodic background. 
While it seems likely that the proof of our main theorems could be adjusted to treat superexponential decay in the $x$-direction, more mild short-range decay seems out of reach.
We also mention the work of Korotyaev-Saburova \cite{korotyaev2017schrodinger}, who studied analogs of strip periodic potentials on more general graphs. For these models, they obtained certain estimates on the locations of the bands.\par 
Our work also fits into the large program of establishing ballistic transport for Schr\"{o}dinger operators with continuous spectrum.
The strong form of ballistic transport defined here is only known for certain classes of operators beyond the periodic regime; results are known for certain limit-periodic operators on $\bbZ$ \cite{Fillman}, and on $\bbR$ \cite{young2021ballistic}. In many other settings, the accessible notions of transport are strictly weaker. We refer the reader to \cite{damanik2010general,damanik2024ballistic} for further background on this and other notions of transport, as well as further works in the one-dimensional, almost periodic setting \cite{KachkovskiyGe, Kachkovskiy, ZhaoZhang, Zhao, Zhaocont}. \par
Besides the aforementioned work of Asch-Knauf and the result of Fillman \cite{Fillman2021}, little is known in multi-dimensional settings, even for weaker forms of ballistic transport. Two results in that vein are the work of Karpeshina et al. \cite{Stolzetall,karpeshina2021ballistic} where ballistic lower bounds for the Abel mean of the
position operator are shown for certain quasi-periodic and limit-periodic operators on $\bbR^2$, and for generic quasi-periodic operators on $\bbR^d$, for $d\geq 2$, respectively, cf. \cite{damanik2024ballistic} and the discussion of Abel average norm-growth ballistic transport. Thus, our work contributes to the understanding of transport beyond one spatial dimension.
\subsection{Overview of the proofs}\label{overviewSection}
The main obstacle in establishing the results \cite{filonov,filonovKlopp,hoang2014absence} is the possibility of eigenvalues embedded in the essential spectrum of the fibered operator $H(k)$. This difficulty is unique to the partially periodic setting because $H(k)$ acts on a cylinder (as opposed to on a torus in the fully periodic setting) and therefore does not have compact resolvent. 
Thus, one needs to understand the variation of these eigenvalues in $k$, and the perturbation theory of embedded eigenvalues is typically quite challenging and must be treated in a context-dependent fashion. The strategy adopted by the above authors is to analytically continue the resolvent on some weighted space from the upper half-plane up to or across the real axis and then to study how this operator behaves as $k$ changes. Per our understanding, these results do not show that the eigenvalues of the fiber operators are analytic functions of $k$ due to the possibility of resonances, i.e., poles of the resolvent that are not eigenvalues. 
The transport properties of fully periodic operators are typically established by showing that the energies are differentiable and non-constant in the quasimomentum, $k$, almost everywhere, see \cite{AschKnauf,Fillman2021}. Thus, we require stronger control on the variation of these embedded energies than has previously been obtained in the works mentioned above.

The advantage of working on $\bbZ^2$ is that the Floquet transform reduces the system to an analytic family of $L$ coupled discrete Schr\"{o}dinger operators, each with a compactly supported potential. Thus, we may examine the eigenvalue problem via transfer matrices. For any fixed $k$, as $x\to \pm\infty$, the space of decaying solutions is finite-dimensional, so we can reduce the existence of an eigenvalue of $H(k)$ to an analytically varying connection problem between these two subspaces across the support of $V$. By forming the determinant associated with this connection problem, we are able to construct the ``partial Bloch variety" of $H$. The theory of analytic varieties allows us to conclude that each eigenvalue may be taken to be analytic almost everywhere (as a function of $k$). This analysis is more complicated than the perturbation theory of matrices because the variables $E$ and $k$ enter this determinant nonlinearly. For this purpose, in the periodic setting, Wilcox \cite{wilcox1978theory} used deep results of Cartan \cite{cartan} and Whitney-Bruhat \cite{WB} on the structure of real analytic varieties, but to keep our work relatively self-contained we develop the necessary machinery ourselves, see in particular, Lemma \ref{analyticLem} and Appendix \ref{WeierstrassPreparation}.
We caution that while this heuristic does suggest that these systems are ``essentially one-dimensional," we emphasize that the usual techniques used in the analysis of Schr\"{o}dinger operators on the line or on a strip $\bbZ\times \{1,\cdots, L\}$ do not apply in any straightforward way. Indeed, the coupling between different vertical slices makes the problem genuinely multi-dimensional, and the unboundedness in the $x$ direction creates the possibility of embedded eigenvalues.

Finally, we mention that this transfer matrix argument applies only on $\bbZ^2$, but we attain analogous transport results for the surface states on $\bbR^d$, $d\geq 3$, in a forthcoming work with Kuchment~\cite{BDKMY}. There, the role of the transfer matrices is played by the Dirichlet-to-Neumann map, which is indeed Fredholm.
\subsection*{Outline of the paper }
This paper is arranged as follows:\par
Section \ref{existenceSection} gives the necessary background on the Floquet transform, as well as some basic scattering results used in later sections. 

In Section \ref{embeddedSection}, we study the analytic variation of the surface states. We begin by studying the spectral properties of the free problem. Then, we use these results to reformulate the full problem using the transfer matrices. Finally, we use this new formulation to prove Theorem \ref{thm:analyticity}. 
In Section \ref{generalTransportSection} we show that Theorem \ref{thm:analyticity} implies Theorem \ref{thm:Z2ballistic} about the ballistic transport of the surface states.

In the next section, Section \ref{scatteringSection}, we turn to the scattering states, and show that a dense set also exhibits ballistic transport. This relies on some more general results on scattering in this setting given in Appendix~\ref{scatteringAppendix}.

In Section \ref{Domains}, we show that $D(\vec{Q})$ is indeed dense in the surface subspace. 

Finally, In Appendix \ref{WeierstrassPreparation}, we provide a form of the Weierstrass preparation theorem for real analytic functions, which is then used to prove the analytic variation of the embedded eigenvalues.

\subsection*{Acknowledgement} We thank Wilhelm Schlag for many helpful discussions about this work. We also thank David Weld for pointing us towards some relevant parts of the photonics literature. 
\section{Preliminaries}\label{existenceSection}
Let $\calB:=[0,2\pi/L)$ be the Brillouin zone. 
For quasimomentum $ k\in \calB $ and $(x,y)\in \bbZ\times\bbZ_L$, we define the \emph{partial Floquet transform} as
\begin{align*}
	(\calU\psi)(k,x,y)=\sum_{m\in \bbZ}^{} \psi(x,y+mL)e^{-ik(y+mL)}.
\end{align*}
Here $\bbZ_L$ may be regarded either as the set $\{0,1,\ldots,L-1\}$ or the integers mod $L$ since all formulas will be invariant mod $L$.
\par
It will be convenient to regard $\ell^2(\bbZ\times \bbZ_L)$ as $\ell^2(\bbZ;\bbC^L)$, where we notate a vector in this space as $\vec{\psi}_x$ for $x\in \bbZ$. 
We will often suppress the $x$ and $y$ variables in $(\calU\psi)(k,x,y)$ to indicate an element of $\ell^2(\bbZ;\bbC^L)$.

For $\{\vec{\psi}_x\}_{x\in \bbZ}$ a sequence in $\ell^2(\bbZ;\bbC^L)$, we write
\begin{align*}
     (H(k)\vec{\psi})_x= \vec{\psi}_{x+1}+\vec{\psi}_{x-1}+\Delta^k\vec{\psi}_x+V_x\vec{\psi}_x,
\end{align*}
where $ \Delta^k $ and $ V_x $ are the linear maps on $ \bbC^L $ given by
\begin{align*}
	&(\Delta^k\vec{\psi}_x)(j)=e^{ik}\vec{\psi}_x(j+1)+e^{-ik}\vec{\psi}_x(j-1)\\
	&(V_x\vec{\psi}_x)(j)=V(x,j)\vec{\psi}_x(j).
\end{align*}
Similarly, we define $H_0(k)$ via
\begin{align*}
    (H_0(k))\vec{\psi}_x=\vec{\psi}_{x+1}+\vec{\psi}_{x-1}+\Delta^k\vec{\psi}_x.
\end{align*}

The standard properties of the Floquet transform, as recorded, for instance, in Section 4 of \cite{Kuchment}, extend readily to the partial Floquet transform:
\begin{proposition}\label{floquetPr}
    The partial Floquet transform $\calU$ has the following properties:
    \begin{enumerate}
        \item $\calU$ is a unitary map
        \begin{align*}
            \calU:\ell^2(\bbZ^2)\rightarrow \int\limits_{\calB}^\oplus \ell^2(\bbZ;\bbC^L)\:\frac{dk}{|\calB|}.
        \end{align*}
        \item We have the unitary equivalence
    	\begin{align*}
    		\calU H \calU^* =\int\limits_{\calB}^\oplus H(k)\:\frac{dk}{|\calB|}.
    	\end{align*}
    \end{enumerate}
\end{proposition}

Let $\Omega:\ell^2(\bbZ^2)\rightarrow\ell^2(\bbZ^2)$ be the wave operator defined by the expression
\begin{align*}
    \Omega =\slim_{t\rightarrow\infty}e^{itH}e^{-itH_0},
\end{align*}
for $\slim$ the strong limit.
We recall that
\begin{align*}
    \calH_{\mathrm{sur}}= \int\limits_{\calB}^\oplus\calH_\textrm{pp}(H(k))\:\frac{ dk}{\abs{\calB} }.
\end{align*}
For later use, we record the following asymptotic completeness result whose content is basically the proof of \cite[Thm 1.8]{frank2003scattering}.
\begin{proposition}\label{Completeness}
    For $H$ strip periodic, we have that $\Omega\psi$ exists for all $\psi\in \ell^2(\bbZ^2)$. In addition, 
    \begin{align*}
        \ell^2(\bbZ^2)=\calH_{\mathrm{sur}}\oplus \Ran (\Omega),
    \end{align*}
    where $H$ has purely absolutely continuous spectrum on $\Ran \Omega$.
\end{proposition}
\begin{proof}
    For each $k\in\calB$, define the wave operator
    \begin{align*}
        \Omega(k)=\slim_{t\rightarrow\infty} e^{itH(k)}e^{-itH_0(k)}
    \end{align*}
    on $\ell^2(\bbZ;\bbC^L)$. The difference between $H(k)$ and $H_0(k)$ is finite rank so the existence and completeness of the wave operators $\Omega(k)$ is an immediate consequence of the Kato-Rosenblum Theorem \cite[Theorem XI.8]{RSVol3}, which requires the difference to be only trace class. Therefore, for each $k\in\calB$,
    \begin{align*}
        \ell^2(\bbZ;\bbC^L)=\calH_{\mathrm{pp}}(H(k))\oplus \Ran \Omega(k).
    \end{align*}

    From \cite[Theorem XIII.85]{RSVol4}, 
    \begin{align*}
        \calU e^{-itH}\calU^*=\int\limits_{\calB}^\oplus e^{-itH(k)}\:\frac{dk}{|\calB|},
    \end{align*}
    and similarly for $e^{-itH_0}$.
    Now observe that for any $\psi\in \ell^2(\bbZ^2)$,
    \begin{align*}
        &\|e^{itH}e^{-itH_0}\psi-\calU^* \int_\calB^\oplus \Omega(k)(\calU\psi)(k)\:\frac{dk}{|\calB|} \|_{\ell^2(\bbZ^2)}^2=\\
        &\quad\int_\calB \|(e^{itH(k)}e^{-itH_0(k)}- \Omega(k))(\calU\psi)(k)\|^2_{\ell^2(\bbZ;\bbC^L)}\:\frac{dk}{|\calB|}.
    \end{align*}
    It follows from the dominated convergence theorem that $\Omega \psi$ exists and is equal to
    \begin{align*}
        \int\limits_{\calB}^\oplus \Omega(k)\:\frac{dk}{\abs{\calB}},
    \end{align*}
    so we may conclude.
\end{proof}

\section{Analytic variation of the surface states}\label{embeddedSection}
\subsection{Spectral theory of $H_0(k)$}\label{FreeSpec}
We start by recording some information about $H_0(k)=\Delta ^k+\Delta^x$, which will be useful in the analysis of $H(k)$.
\begin{proposition} 
    Let $ v_n\in \bbC^{L} $ be the vector $ v_n(j)=\frac{1}{\sqrt{L} }\zeta^{jn} $ for $ \zeta=e^{2\pi i/L} $. Then $ \{v_n\}_{n\in \bbZ_L} $ is an orthonormal eigenbasis of $ \Delta^k $ with associated eigenvalues $ 2\cos(k+2\pi n/L) $.
\end{proposition}
\begin{proof}
    We compute
    \begin{align*}
        (\Delta^kv_n)(j)=&(1/\sqrt{L} )(e^{i(k+2(j+1)n\pi/L)} +e^{-i(k-2(j-1)n\pi/L)})=2\cos(k+2\pi n/L)\zeta^{jn}/\sqrt{L} 
    \end{align*}
    which shows that $ v_n $ is an eigenvector, as claimed. The fact that these vectors are linearly independent is not immediate because two of the expressions for the eigenvalues, $ 2\cos(k+2\pi n/L) $, may be equal. However, it follows from the non-singularity of the Vandermonde matrix. Indeed, we may compute
    \begin{align*}
        \det(v_1\cdots v_L)=\frac{1}{L^{L/2}}\prod_{1\leq m<\ell<L}(\zeta^m-\zeta^\ell)\ne 0
    \end{align*}
    to conclude.
\end{proof}

Forming the vector $ \Psi_x=\vec{\psi}_x\oplus\vec{\psi}_{x-1}$ in $ \bbC^L\oplus\bbC^L $, the equation $H_0(k)\psi=E\psi$ may be written via transfer matrices as
\begin{align*}
    \Psi_{x+1}=T_0(E,k)\Psi_x,
\end{align*}
where 
\begin{align*}
    T_0(E,k)\Psi_x=\left( (E-\Delta^k)\vec{\psi}_x-\vec{\psi}_{x-1} \right) \oplus\vec{\psi}_x,
\end{align*}
or as a block matrix:
\begin{align*}
    &T_0(E,k)=\begin{pmatrix}
        E-\Delta^k & -\Id \\
        \Id & 0
    \end{pmatrix}.
\end{align*}
Observe that if $\vec{\psi}_x=a_x v_j$ for all $x\in\bbZ$, then the vector-valued difference equation $H_0(k)\psi=E\psi$ reduces to the scalar difference equation
\begin{align*} 
    a_{x+1}=e_j a_x-a_{x-1},
\end{align*}
where $e_j =E- 2\cos(k+2\pi j/L) $.\par
Thus, decomposing $T_0$ with respect to the subspaces $\calV_j=\text{span}\{v_j\oplus 0,0\oplus v_j\}\subset \bbC^L\oplus \bbC^L$ we see that, up to unitary conjugation, $T_0(E,k)$ is given by
\begin{align}\label{T0Decomp}\bigoplus_{j\in\bbZ_L}\begin{pmatrix}
        e_j & -1 \\
        1 & 0
    \end{pmatrix}.
\end{align}
    
Each summand is a 1d transfer matrix of the form
\begin{align*}
    \begin{pmatrix}
        w & -1 \\
        1 & 0
    \end{pmatrix},
\end{align*}
whose eigenvalues are parameterized by the \emph{Joukowsky map} $J:\overline{\mathbb C}\to \overline{\mathbb C}$, which we now recall. It is defined by $J(z)=z+\frac{1}{z}$ and admits analytic inverses $\mu^-:\bbC\setminus [-2,2]\rightarrow \bbC\setminus \overline{\bbD}$ and $\mu^+:\bbC\setminus[-2,2]\rightarrow \bbD$ where $\bbD$ is the unit disc. By solving $J(z)=w$, we may write these functions explicitly as
\begin{align*}
    \mu^\pm(w)=\frac{w\mp\sqrt{w^2-4}}{2}.
\end{align*}
The root is unambiguous because $\sqrt{z^2-4}$ has two analytic branches on $\bbC\setminus[-2,2]$ with images lying in either $\bbD$ or $\bbC\setminus \overline{\bbD}$. When $w\in [-2,2]$, we will also write $\mu^\pm(w)$ to mean the two (not necessarily unique) solutions of $J(z)=w$, but for these $w$'s all claims will be symmetric in $+$ and $-$ so we do not fix a convention. In this case, $\mu^\pm(w)$ both lie on $\partial\bbD$.\par
The eigenvector associated to $\mu^\pm(w)$ is given by $\begin{pmatrix}
    \mu^{\pm}(w)\\1
\end{pmatrix}$. Clearly then, for $w\not\in[-2,2]$, an initial condition for the difference equation corresponding to the 1d transfer matrix is decaying at $+\infty$ if and only if it is in the eigenspace of $\mu^+(w)$ and similarly at $-\infty$. On the other hand, when $w\in[-2,2]$, there are no such solutions in either direction.\par
With this in mind, we associate to any $(E,k)\in\bbR^2$ the subspaces $\calV^\pm(E,k)\subset \bbC^L\oplus \bbC^L$ of vectors which decay at $\pm\infty$ under repeated application of $T_0$. If $j\in\bbZ_L$ is such that $e_j(E,k)\not\in[-2,2]$, then let $\calV_j^\pm(E,k)\subset \bbC^{L}\oplus \bbC^{L}$ be the eigenspaces corresponding to $\mu^\pm$. The decomposition (\ref{T0Decomp}) and the above analysis show that $\calV^\pm$ is given by
\begin{align*}
    \calV^\pm(E,k)=\bigoplus_{\{j\in\bbZ_L\mid e_j(E,k)\not\in[-2,2]\}}\calV_j^\pm(E,k).
\end{align*}
Each subspace $\calV_j^\pm$ depends on $E$ and $k$ through the quantity $\mu^\pm(e_j(E,k))$, where we recall that  $\mu^\pm$ is an analytic function on $\bbR\setminus [-2,2]$. Therefore, each subspace $\calV_j^\pm$ varies analytically in $E$ and $k$ (in the sense that the associated projector is an analytic operator) inside the open set $\{(E,k)\in \bbR^2\mid e_j(E,k)\not\in [-2,2]\}$. \par
We note that the indices summed to obtain $\calV^\pm$ only change when for some $j\in \bbZ_L$ we have $e_j=\pm 2$. It follows that the total subspaces $\calV^\pm$ each vary analytically away from the curves given by $e_j(E,k)=2$ and $e_j(E,k)=-2$, for each $j\in \bbZ_L$, across which the dimension of $\calV^\pm$ may jump. Let $\calA$ be the union of these curves, i.e.
\begin{align}\label{curves}
    \calA=\bigcup_{j\in\bbZ_L}\{(E,k)\in\bbR\times \calB\mid e_j(E,k)= 2\}\cup\{(E,k)\in\bbR^2\mid e_j(E,k)= -2\},
\end{align}
and let the open set $\calA^c$ be their complement inside $\bbR\times\calB$. 

Finally, for $(E,k)\in\bbR\times \calB$ let $\calI^\pm(E,k):\calV^\pm\hookrightarrow \bbC^L\oplus\bbC^L$ be the inclusion of $\calV^\pm$ in the full space and $\calP^\pm(E,k):\bbC^L\oplus\bbC^L\rightarrow\calV^\pm\subset \bbC^L\oplus\bbC^L$ be the corresponding orthogonal projection.\par
In summary, we have proven the following proposition: 

\begin{proposition}\label{T0Summary}
	In the above notation: 
 \begin{itemize}
     \item The eigenvalues of $T_0(E,k)$ are given by $\{\mu^\pm(e_j(E,k)\mid j\in\bbZ_L\}$.
     \item A vector $\psi_0\in\bbC^L\oplus \bbC^L$ satisfies $\{T_0(E,k)^{\pm n}\psi_0\}_{n=0}^\infty\in\ell^2(\bbN)$
     if and only \\
     if $\psi_0\in\Ran\calP^\pm(E,k)$. 
     \item The operators $\calI^\pm(E,k)$ and $\calP^\pm(E,k)$ are analytic in $\calA^c\subset\bbR\times \calB$.
 \end{itemize}
 \end{proposition}
    
\subsection{Eigenvalue problem for $H(k)$}\label{Reformulation}
Now we turn our attention to the eigenvalue problem for the full operator $H(k)$ for some fixed $k\in\calB$. Using the transfer matrix formalism developed above, we will reduce the eigenvalue problem to a connection problem across the support of the potential,  as detailed in Lemma \ref{mainLemma} below.\par 
First, write $H(k)\psi=E\psi$ as the vector-valued difference equation
\begin{align}\label{diffEq}
	\vec{\psi}_{x+1}=(E-\Delta^k-V_x)\vec{\psi}_x-\vec{\psi}_{x-1},
\end{align}
where $V_x=V(x,\cdot)$ is as described above. As above, using $ \Psi_x=\vec{\psi}_x\oplus\vec{\psi}_{x-1}$ in $ \bbC^{L}\oplus \bbC^{L} $, we can write (\ref{diffEq}) as
\begin{align*}
	\Psi_{x+1}=\begin{pmatrix}
	    E-\Delta^k-V_x&-\Id\\
     \Id & 0
	\end{pmatrix}\Psi_x.
\end{align*}
Now let 
\begin{align}\label{TVFactor}
    T_V=\begin{pmatrix}
	    E-\Delta^k-V_{R}&-\Id\\
     \Id & 0
	\end{pmatrix}\begin{pmatrix}
	    E-\Delta^k-V_{R-1}&-\Id\\
     \Id & 0
	\end{pmatrix}\cdots\begin{pmatrix}
	    E-\Delta^k-V_{-R}&-\Id\\
     \Id & 0
	\end{pmatrix}
\end{align}
be the transfer matrix from the left side of the support of $V$ to the right. \par
With this in hand, we arrive at the key lemma:
\begin{lemma}\label{mainLemma}
    Let $\tilde{\calP}^+=\Id - \calP^+$. Then $E\in\bbC$ is an eigenvalue of $H(k)$ if and only if the matrix
    \begin{align*}
        A(E,k)=(\tilde{P}^+T_V\calI^-)(E,k)
    \end{align*}
    has nontrivial kernel. Furthermore, the eigenvalues of $H(k)$ are given by the zeroes of 
    \begin{align*}
        F(\cdot,k)=\det(A^*A(\cdot,k))
    \end{align*}
    counted with multiplicity.
\end{lemma}
\begin{proof}
    From Proposition \ref{T0Summary}, we see that a solution to $H(k)\psi=E\psi$ is $\ell^2$ at $-\infty$ if and only if $\psi_{-R}\in\Ran \calI^-$ and $\psi_R\in\Ran \calI^+$. Therefore, $E$ is an eigenvalue if and only if $T_V$ sends a vector in $\Ran \calI^-$ to a vector in $\Ran\calP^+$. This is precisely the condition that $A(E,k)$ is singular, since the projector $\tilde{P}^+$ enforces that the output vector has no non-decaying component. Since $F$ is the product of the singular values of $A$, the second claim follows as well.
\end{proof}
\subsection{Proof of Theorem~\ref{thm:analyticity}}\label{AnalyticDisc}
We require the following lemma on the analytic variation of eigenvalues of $H(k)$:
\begin{proposition}\label{analyticLem}
    Let $E_0\in \bbR$ be an eigenvalue of $H(k_0)$ for some $k_0\in\calB$, and suppose that $(E_0,k_0)\not\in \calA$. 
    Then there exists open intervals $I\subset \bbR$ containing $E_0$ and $J\subset \calB$ containing $k_0$ such that one of the following possibilities hold:
    \begin{enumerate}
        \item The point $(E_0,k_0)$ is an isolated point of the spectrum, i.e.  it is the only point $(E,k)\in I\times J$ such that $E_0$ is an eigenvalue of $H(k)$.
        \item There exists countably many intervals $J_i\subset J$ and analytic functions $f_i:J_i\rightarrow \bbR$ such that for some finite set $\calN$
        \begin{align*}
        \{(E,k)\in I\times (J\setminus \calN) \mid \ker(H(k)-E)\neq 0\}=\bigcup_i\{(f_i(k),k)\mid k\in I_i\}.
        \end{align*}
    \end{enumerate}
\end{proposition}

\begin{proof}
    First, note that $E\mapsto F(E,k_0)$ cannot vanish identically because then $H(k)$ would have an interval of eigenvalues. 
    Thus, we may apply Lemma \ref{WeierstrassPreparationLemma}  to find some neighborhood of $(E_0,k_0)$ of the form $I\times J$ on which the zero set of $F$ coincides with the zero set of a polynomial $g(E,k)$ with a discriminant $D(k)$ that is not identically $0$. 
    We have, by the definition of the discriminant, that
    \begin{align*}
    \left\{ k\in J:D(k)= 0\right\}=\left\{k\in J: \exists E\in \mathbb R\;\text{s.t.}\; \frac{\partial g}{\partial E}(E,k)=g(E,k)= 0\right\}.  
    \end{align*} 
    This set must be countable and accumulate at the endpoints of $J$ because $D(k)$ is analytic so we may take it to be finite by shrinking $J$ if necessary.
    If the only points $(E,k)\in I\times J$ such that $g(E,k)=0$  are for $k$ with  $D(k)=0$ then the first possibility in the conclusion holds.
    Otherwise, for any $k\in J$ for which $D(k)\neq 0$, we may apply the analytic implicit function theorem to conclude.
\end{proof}

Now, we establish the non-constancy of the energies:
 
\begin{lemma}\label{lem:nonconstancy}
    For any $E_0\in\bbR$, the set of $k\in\calB$ for which $E_0$ is an eigenvalue of $H(k)$ has measure $0$.
\end{lemma}
\begin{proof}
    The set of $k\in\calB$ such that $(E_0,k)$ is in $\calA$ is finite (as these are non-constant curves) so it suffices to consider $k$ such that $(E_0,k)\in\calA^c$. Fix such $k_0$ and let $I$ be an interval containing $k_0$ so that $\{E_0\}\times I\subset\calA^c$. We will show that $F(E_0,k)$ vanishes only on a set of measure $0$ in $I$ or equivalently that $A(E_0,k)$ is only singular on a set of measure $0$.\par
    Now recall that $A(E_0,k)$ depends on $k$ through the expressions $e_j(E_0,k)$, which appears in $T_V$ and $\mu^\pm(e_j(E_0,k))$, the latter of which appears in $\calI^+$ and $\tilde{\calP}^+$ for some subset of $j$ in $\bbZ_L$. Since
    \begin{align*}
        e_j(E_0,k)=E_0-2\cos(k+it+2\pi j/L),
    \end{align*}
    and the Joukowsky maps $\mu^\pm(\cdot)$ admit analytic extensions to $\bbC\setminus[-2,2]$, we may analytically extend $A(E_0,k)$ to the vertical strip $\bbV=\{z\mid\Re{z}\in I\}$ so long as the image of this strip under each $e_j$, $j\in\calI$, avoids $[-2,2]$. For $k,t \in \bbR$, $e_j(E_0,k+it)$ has imaginary part
    \begin{align*}
        \sin(k+2\pi j/L)\sinh(t).
    \end{align*}
    Thus, unless $k+2\pi j/L\in \pi \bbZ$, we have that $e_j(E_0,k+it)$ has non-zero imaginary part for $t\neq 0$, and therefore avoids $[-2,2]$. We see then that, except at finitely many points, we may extend $A(E_0,k)$ to an analytic family of operators on $\bbV$.\par
    Now, since $A(E_0,k)$ is an analytic family on $\bbV$, to show that it is singular on a set of measure $0$ in $I$, it suffices to show that it does not vanish identically on $\bbV$. For this, we use the structure of $A(E_0,k)$. Recalling that $A(E,k)=\tilde{\calP}^+T_V\calI^-$, by expanding the product (\ref{TVFactor}), we may write
    \begin{align*}
        A(E,k)=\tilde{\calP}^+(T_0)^{2R+1}\calI^-+\tilde{\calP}^+\tilde{A}\calI^-,
    \end{align*}
    where $\tilde{A}$ is the sum of products, each of which contains at most $2R$ copies of $T_0$ in addition to terms of the form $V_x$, $|x|\leq R$. Since the eigenvalues of $T_0(E_0,k_0+it)$ are given by
    \begin{align*}
        \mu^\pm(E_0-2\cos(k_0+it+2\pi j/L)),
    \end{align*}
    we see that for large $t>0$, the exponential growth of cosine off the real axis ensures that there exist constants $C>0$ and $c>0$ such that
    \begin{align*}
        \op{T_0(E_0,k_0+it)}_{\textrm{op}}\leq Ce^{ct},
    \end{align*}
    and consequently, for some other $C>0$
    \begin{align*}
        \|\tilde{A}(E_0,k_0+it)\|_\textrm{op}\leq Ce^{c2Rt},
    \end{align*}
    by virtue of the boundedness of $V$.\par
    Now observe that the image of $\calI^-$ consists of eigenspaces of $T_0$ corresponding to $\mu^-(e_j)$ for $j\in\calI$ and furthermore that these eigenspaces lie in the image of $\tilde{\calP}^+$ because, by construction $\tilde{\calP}^+$ corresponds to the \emph{complement} of the eigenvalues $\mu^+$. In other words $\tilde{\calP}^+\calI^-=\calI^-$. The form of $\mu^+$ and the growth of cosine then easily show that for large $t$ there exists $C'>0$ so that
    \begin{align*}
        \abs{\mu^+(e_j(E_0,k_0+it))}>C'e^{ct}.
    \end{align*} 
    Therefore, for any $v\in\bbC^L\oplus \bbC^L$ and $t$ sufficiently large,
    \begin{align*}
        \|\tilde{\calP}^+(T_0)^{2R+1}\calI^-(E_0,k_0+it)v\|> C'e^{c(2R+1)t}\|v\|,
    \end{align*}
    so we have for large $t>0$
    \begin{align*}
        &\|A(E,k_0+it)\|_{\textrm{op}}=\|\tilde{\calP}^+(T_0)^{2R}\calI^-+\tilde{\calP}^+\tilde{A}\calI^-\|_{\textrm{op}}\\
        &\geq \|\tilde{\calP}^+(T_0)^{2R}\calI^-\|_{\textrm{op}} -\|\tilde{\calP}^+\tilde{A}\calI^-\|_{\textrm{op}}>C'e^{c(2R+1)t}-Ce^{c2Rt}
    \end{align*}
    and we conclude that $A$ is non-singular for $t>0$ sufficiently large, which completes the proof.
\end{proof}

With Lemma~\ref{analyticLem} and Lemma~\ref{lem:nonconstancy}, we may now prove Theorem~\ref{thm:analyticity}.

\begin{proof}[Proof of Theorem~\ref{thm:analyticity}]	
    We will show that there exists countably many open sets $U_l\subset \calB$, countably many analytic functions $\lambda_{l,i}:U_{l}\to \bbR$, and associated eigenprojectors
    \begin{align*}
    	k\mapsto \pi_{l,i}(k):\ell^{2}(\bbZ;\bbC^{L})\to \ell^{2}(\bbZ;\bbC^{L})
    \end{align*}
    such that the following holds:
    for almost every $k\in \calB$ such that $\calH_{\mathrm{pp}}(H(k))$ is not empty and
    \begin{align*}
        H(k)\vert_{\calH_{\mathrm{pp}}(H(k))}=\sum_{U_l\ni k}\sum_{i}\lambda_{l,i}(k)\pi_{l,i}(k).
    \end{align*}
    Since the eigenprojector associated with an analytically varying eigenvalue of an analytically varying matrix is analytic up to a discrete set, this will justify items~(\ref{item:thm1}) and  (\ref{item:thm2}) of the theorem. To prove this for $k$ such that $(E,k)\not\in\calA$ for all $E$, we simply apply Lemma~\ref{analyticLem} locally. 
    On the other hand, for $k$ with $(E,k)\in \calA$ for some  $E$, we recall that $\calA$ is defined by the union  over $j$ of 
    \begin{align*}
    	\{(2\cos(k + 2\pi j /L)\pm 2,k) \mid k\in \calB\},
    \end{align*}
    and that $\calV^{\pm}(E,k)$ are piecewise constant on these curves.
    Therefore, for each $j$
    \begin{align*}
        k\mapsto F(2\cos(k+2\pi j/L)\pm 2,k)
    \end{align*}
    are piecewise analytic functions and so they vanish on a (possibly empty) union of subintervals of $\calB$ and discrete points. 
    Adding the functions $k\mapsto 2\cos(k+2\pi j /L)\pm 2$ on each possible interval of vanishing to our collection of $\lambda_{l,i}$, we obtain the claim.\par
    Finally, (\ref{item:thm3}) is a direct consequence of Lemma~\ref{lem:nonconstancy}, so we are done.
\end{proof}

\section{Directional ballistic transport for the surface states} \label{generalTransportSection}
We now prove Theorem~\ref{thm:Z2ballistic}, by analyzing propagation in the $y$ and $x$ directions in turn. Recall that $Y_H(T)=e^{iTH}Ye^{-iTH}$ is the Heisenberg-evolved position operator in the $y$-directions.

Define the momentum operators $P^x:=i[H,X]$ and $P^y:=i[H,Y]$, which can be written explicitly as
\begin{align*}
    &(P^x\psi)(x,y)=i(\psi(x+1,y)-\psi(x-1,y))\\
    &(P^y\psi)(x,y)=i(\psi(x,y+1)-\psi(x,y-1)).
\end{align*}
We prove:
\begin{proposition}\label{YTransportThm}
    For any $\psi\in D(Y)\cap \calH_{\mathrm{sur}}\setminus\{0\}$, the asymptotic velocity
    \begin{align}\label{BTeq}
        \lim_{T\to\infty}\frac{1}{T}Y_H(T)\psi
    \end{align}
    exists and is non-zero. 
\end{proposition}
\begin{proof}
    Making use of the identity
    \begin{align*}
        Y_H(T)\psi =Y_H(0)\psi+\int_{0}^T P^y_H(t)\psi \:dt,
    \end{align*}
    valid for $\psi\in D(Y)$ by the proof of \cite[Thm. 2.1.]{damanik2015quantum} (see also \cite[Prop. 3.1]{Fillman2021}), it is enough to show that  $\lim\limits_{T\rightarrow\infty}\frac 1T\int_{0}^TP_H^y(t)\psi\: dt$ exists and is non-zero.
    
    Since the map
    \begin{align*}
        \psi\mapsto \frac1T\int_0^TP_H^y(t)\psi\:dt
    \end{align*}
    is bounded on $\ell^2(\bbZ^2)$ independently of $T$
    it suffices to establish the theorem for $\psi$ of the form
    \begin{align}
        \calU\psi=\left(\sum_{l=1}^N\int\limits_{U_l}^\oplus\sum_{n=1}^M\pi_{l,n}(k)\:\frac{dk}{\abs{\calB}}\right)\calU\psi
    \end{align}
    for some $N,M>0$. 
    
    Introducing  $\calP_\textrm{pp}(k)$ and $\calP_{\textrm{c}}(k)$, the projectors onto the pure point and continuous subspaces of $H(k)$, we write
    \begin{align*}
        &\calU\left(\frac1T\int_0^T P^y_H(t)\psi)\:dt\right)=\frac1T \int_0^T  e^{itH(k)}P^y(k)e^{-itH(k)} (\calU\psi)(k)\:dt\\
        &=\frac1T \int_0^T  e^{itH(k)}\calP_{\textrm{pp}}(k)P^y(k)e^{-itH(k)} (\calU\psi)(k)\:dt+\frac1T \int_0^T  e^{itH(k)}\calP_{\textrm{c}}(k)P^y(k)e^{-itH(k)} (\calU\psi)(k)\:dt\\
        &:= A(T) + B(T),
    \end{align*}    
    where we set
    \begin{align*}
        (P^y(k)\psi)(k,x,y)=i(e^{ik}\psi(k,x,y+1)-e^{-ik}\psi(k,x,y-1)).
    \end{align*}
    We will show below that
    \begin{align}\label{eq:ALim}
        \lim_{T\rightarrow\infty}A(T) =\left(\sum_{l=1}^N \int\limits_{U_l}^\oplus \sum_{n=1}^{M} \pi_{l,n}(k)P^y(k)\pi_{l,n}(k)\:\frac{dk}{|\calB|}\right)\calU\psi,
    \end{align}
    that the right-hand side is non-zero and 
    that $B(T)\rightarrow 0$.
    The proof of these first two facts is an argument from \cite{AschKnauf}, but for completeness, we add the details below.
    We note that
    \begin{align*}
        A(T)=\frac1T\int \limits_0^T\left(\sum_{l=1}^N \int\limits_{U_l}^\oplus \sum_{m=1}^{M_l}\sum_{n=1}^{M} e^{it(\lambda_{l,m}(k)-\lambda_{l,n}(k))}\pi_{l,m}(k)P^y(k)\pi_{l,n}(k)\:\frac{dk}{|\calB|}\right) \calU\psi\:dt,
    \end{align*}
    where $M_l$ is the (possibly infinite) number of branches over $U_l$.
    We then have
    \begin{align*}
        &\bigg\|A(T)-\left(\sum_{l=1}^N \int\limits_{U_l}^\oplus\sum_{n=1}^{M} \pi_{l,n}(k)P^y(k)\pi_{l,n}(k)\,\frac{dk}{|\calB|}\right)\calU\psi\bigg\|^2_{\ell^2(\bbZ;\bbC^L)}\\
        &\quad=\sum_{l=1}^N \int\limits_{U_l} \sum_{m=1}^{M_l}\left\|\pi_{\ell,m}(k)\sum_{\substack{n\ne m\\n=1}}^{M} \frac1T\int\limits_0^Te^{it(\lambda_{l,m}(k)-\lambda_{l,n}(k))}dtP^y(k)\pi_{\ell,n}(k)\calU\psi(k,\cdot,\cdot) \right \|^2_{\ell^2(\bbZ;\bbC^L)}\,  \frac{dk}{|\calB|}
    \end{align*}
    by the Pythagorean theorem. We observe that by the definition of the $\lambda_{l,n}(k)$, the integrand is $O(1/T)$ for each $k$. So,
    we may conclude via the dominated convergence theorem that \eqref{eq:ALim} holds.\par
    To see that the limit is non-zero, we note that for each $n$, $\lambda_{l,n}'(k)$ exists at all but finitely many points of $U_l$ and is non-zero. The claim then follows from the identity
    \begin{align*}
        \pi_{l,n}(k)P^y(k)\pi_{l,n}(k)=\frac{1}{2}\lambda_{l,n}'(k)\pi_{l,n}(k),
    \end{align*}
    which may be derived by differentiating $\pi_{l,n}(k)H(k)\pi_{l,n}(k)$.
    
    Finally, we show the limit of $B(T)$ is $0$. We have
    \begin{align*}
        \|B(T)\|^2=\sum_{l=1}^N \int\limits_{U_l} \left\|\sum_{n=1}^{M}\frac1T\int\limits_0^Te^{it(H(k)-\lambda_{l,n}(k))}\: dt\,\calP_{\mathrm{c}}(k)P^y(k)\pi_{l,n}(k)(\calU\psi)(k)\right\|^2_{\ell^2(\bbZ;\bbC^L)} \,  \frac{dk}{|\calB|},
    \end{align*}
    and denoting $\tilde{\psi}_{\ell,n}(k)=P^y(k)\pi_{\ell,n}(k)(\calU\psi)(k)$, we can write the integrand of each summand as
    \begin{align*}
        &\sum_{n,m=1}^{M}\frac{1}{T^2}\left\langle \int\limits_0^Te^{it(H(k)-\lambda_{l,n}(k))}\mathrm dt\,\calP_{\mathrm{c}}(k)\tilde{\psi}_{l,n}(k),\int\limits_0^Te^{is(H(k)-\lambda_{l,m}(k))}\calP_{\mathrm{c}} (k)ds\,\tilde{\psi}_{l,m}(k)\right\rangle_{\ell^2(\bbZ;\bbC^L)}\\
        &=\sum_{n,m=1}^{M}\frac{1}{T^2}\iint\limits_{[0,T]\times [0,T]}e^{i(t\lambda_{l,n}(k)-s\lambda_{l,m}(k))}\braket{\calP_{\mathrm{c}}(k)\tilde{\psi}_{l,n}(k),e^{iH(k)(s-t)}\calP_{\mathrm{c}}(k)\tilde{\psi}_{\ell,m}(k)}\, ds\,dt.
    \end{align*}
    Now, use the spectral theorem to see that
    \begin{align*}
        &\frac{1}{T^2}\iint\limits_{[0,T]\times[0,T]}e^{i(s\lambda_{l,n}(k)-t\lambda_{l,m}(k))}\braket{\calP_{\mathrm{c}}(k)\tilde{\psi}_{l,n}(k),e^{iH(k)(s-t)}\calP_{\mathrm{c}}(k)\tilde{\psi}_{l,m}(k)}\, ds\, dt\\
        &=\frac{1}{T^2}\iint\limits_{[0,T]\times[0,T]}e^{i(s\lambda_{l,n}(k)-t\lambda_{l,m}(k))}\,ds\,dt\int\limits _{\bbR}  e^{i(s-t)\lambda}\mu(d\lambda)\, ds\, dt,
    \end{align*}
    where $ \mu $ is the spectral measure of $ \calP_{\mathrm{c}}(k)\tilde{\psi}_{\ell,n}(k)$ and $\calP_{\mathrm{c}}(k)  \tilde{\psi}_{\ell,m}(k) $, which is, due to the projections, a continuous measure. As in the proof of Wiener's theorem, we use Fubini's theorem to rewrite the integral as
    \begin{align*}
        \int\limits_{\bbR}\frac{1}{T^2}\iint\limits_{[0,T]\times[0,T]}e^{is(E_{\ell,n}-\lambda)}e^{it(\lambda-E_{\ell,m})} \,ds\,dt\,\mu(d\lambda)
    \end{align*}
    and observe that the inner integral goes to $ \chi_{\{0\} }(\lambda_{\ell,n}-\lambda)\chi_{\{0\} }(\lambda_{\ell,m}-\lambda) $. Since the integrand is uniformly bounded, we may use the dominated convergence theorem to find that for $k$
    \begin{align*}
        &\lim_{T\to\infty}
        \int\limits_{\bbR}\frac{1}{T^2}\sum_{n,m=1}^{M}\iint\limits_{[0,T]\times[0,T]}e^{is(\lambda_{l,n}-\lambda)}e^{it(\lambda-\lambda_{l,m})} \,ds\,dt\,\mu(d\lambda)\\
        &=\mu(\{\lambda_{l,n}\}\cap\{\lambda_{l,m}\})=0.
    \end{align*}
    Applying dominated convergence to the original integral for $\|B(T)\|^2$ allows us to conclude.
\end{proof}
We will now prove the following result regarding the absence of transport in the $x$-direction. 
\begin{proposition}\label{XAbsenceThm}
     For any $\psi \in  D(X)\cap \calH_\mathrm{sur}$ the asymptotic velocity in the $x$-direction vanishes, i.e.
    \begin{align*}
        \lim_{T\rightarrow\infty} \frac{1}{T}X_H(T)\psi=0.
    \end{align*}
\end{proposition}
First, we prove a technical lemma:
 \begin{lemma}\label{TechLemmaApp}
    Let $\varphi_n$ and $\varphi_m$ be eigenfunctions of $H(k)$ of eigenvalue $\lambda_n$ and $\lambda_m$, respectively. Then for any bounded operator $A:\ell^2(\bbZ\times \bbZ_L)\rightarrow \ell^2(\bbZ\times \bbZ_L)$, we have that
    \begin{align*}
        \lim\limits_{T\rightarrow \infty}\frac{1}{T^2}\iint\limits_{[0,T]\times[0,T]}e^{i(s\lambda_n-t\lambda_m)}\braket{A\varphi_n ,e^{i(t-s)H(k)}A \varphi_m}\,  dt\, ds=0.
         \end{align*}
 \end{lemma}
 \begin{proof}
     We expand $A\varphi_n$ as
    \begin{align*}
         A\varphi_n =\sum_{\ell=1}^\infty a_{n,\ell}\varphi_\ell+\calP_{\textrm{c}}A\varphi_n
    \end{align*}
    for $ a_{n,\ell}=\braket{\varphi_n, A\varphi_\ell} $, and $\calP_{\textrm{c}}$ the projection to the continuous subspace of $H(k)$.
    This allows us to obtain
    \begin{align} \label{pSplit}
        \begin{split}
        	&\frac{1}{T^2}\iint\limits_{[0,T]\times[0,T]}e^{i(s\lambda_n-t\lambda_m)}\braket{A\varphi_n ,e^{i(t-s)H(k)}A \varphi_m}\,  dt\, ds\\
        	&\quad=\frac{1}{T^2}\sum_{\ell=1}^{\infty}a_{n,\ell}a_{m,\ell} \iint\limits_{[0,T]\times[0,T]}e^{-it(\lambda_\ell-\lambda_n)}e^{-is(\lambda_m-
            \lambda_\ell)}\,dt\,ds\\
            &\qquad+\frac{1}{T^2}\iint\limits_{[0,T]\times[0,T]}e^{i(s\lambda_n-t\lambda_m)}\braket{\calP_{\textrm{c}}A\varphi_n,e^{i(t-s)H}\calP_{\textrm{c}}A\varphi_m}\:dt\,ds
        \end{split}
    \end{align}
    because the cross-terms vanish by orthogonality. Now we consider each summand separately. \par
    For the first summand, if either  $\lambda_n=\lambda_\ell$ or $\lambda_m=\lambda_\ell$, we may use Lemma 2.3 of \cite{simon1990absence} to conclude that $ a_{n,\ell}=0$ or $a_{m,\ell}=0 $ respectively. If $\lambda_n\neq \lambda_\ell$, we get that the $t$ integral is $O(1)$, and the integral over $s$ is $O(T)$, which gives the desired decay. This shows that for each $\ell$ the summand goes to $0$. Since the integral is bounded by $ 1 $ and 
    \begin{align*}
        \sum_{\ell=1}^{\infty} \abs{a_{n,\ell}a_{m,\ell}} \leq \sqrt{\sum_{\ell=1}^{\infty} \abs{a_{\ell,m}}^2} \cdot\sqrt{\sum_{\ell=1}^{\infty} \abs{a_{\ell,m}}^2} \leq \|A\varphi_n\|\|A\varphi_m\|,
    \end{align*}
    we may use the dominated convergence theorem to pass the limit in $ T $ under the sum in (\ref{pSplit}) to see that it is $ O(\frac{1}{T}) $.\par
    The second term goes to $0$ by the proof that $\lim_{T\rightarrow\infty}B(T)=0$ in Proposition~\ref{YTransportThm}.
 \end{proof}
Now we can prove Proposition~\ref{XAbsenceThm}.
\begin{proof}[Proof of Proposition~\ref{XAbsenceThm}]
    As in the proof of Proposition~\ref{YTransportThm}, it is enough to show that 
    \begin{align*}
    \frac1T\int_{0}^TP_H^x(t)\psi \:dt\rightarrow 0
    \end{align*}
    in $\ell^2(\bbZ^2)$ as $T\to\infty$.
    Note that $P^x$ is bounded and it respects the $y$-geometry, i.e., $\calU P^x \psi=P^x \calU\psi$. Thus, by Plancherel and Fubini we may write 
    \begin{align*}
        \left\|\frac1T\int\limits_{0}^TP_H^x(t)\psi \:dt\right\|_{\ell^2(\bbZ^2)} &=\left\|\int\limits_\calB \frac1T\int\limits_0^T P^x_{H(k)}(t) (\calU\psi)(k)\:dt \frac{dk}{|\calB|}\right\|_{\ell^2(\bbZ\times \bbZ_L)}\\
        &\leq \int\limits_\calB \left\|\frac1T \int\limits_0^T P^x_{H(k)}(t)(\calU\psi)(k)\:dt\right\|_{\ell^2(\bbZ\times \bbZ_L)}\:\frac{dk}{|\calB|}.
    \end{align*}
    By dominated convergence, it suffices to show that for each $k$,
    \begin{align*}
        \lim_{T\rightarrow\infty}\left\|\frac1T\int\limits_0^T P^x_{H(k)}(t)(\calU\psi)(k)\right\|_{\ell^2(\bbZ\times\bbZ_L)}=0.
    \end{align*}
    
    For some fixed $k$, we write
    \begin{align*}
        (\calU\psi)(k)=\sum_{\alpha}c_\alpha \varphi_\alpha,
    \end{align*}
    where the $\varphi_\alpha$ are an orthonormal basis of eigenfunctions of $H(k)$, with eigenvalues $\lambda_\alpha$.
    By a limiting argument, it suffices to assume that the sum is finite.
    
    Now we write
    \begin{align*}
    \left\|\frac1T P^x_{H(k)}(t)(\calU\psi)(k)\right\|_{\ell^2(\bbZ\times\bbZ_L)}^2=
    \frac{1}{T^2}\int\limits_{0}^T\int\limits_{0}^T\braket{P_{H(k)}^x(s)(\sum_\alpha c_\alpha \varphi_\alpha),P_{H(k)}^x(t)(\sum_\beta c_\beta \varphi_\beta)}\:dt ds\\
    =\sum_{\alpha,\beta}\frac{c_\alpha\overline{c}_\beta}{T^2}\iint\limits_{[0,T]^2}e^{i(is\lambda_\alpha-it\lambda_\beta}\braket{P^x\varphi_\alpha,e^{i(t-s)H(k)}\varphi_\beta}\:dsdt
        \end{align*}
    The conclusion now follows from Lemma~\ref{TechLemmaApp}.
\end{proof}

Actually, up to interchanging symbols, the above argument immediately proves the following result:
\begin{theorem}\label{Thm:Simon}
    Let $H=\Delta +V$ be a Schrödinger operator on $\bbZ^m$ for any $m\geq 1$ with $V$ real-valued and bounded. Then for any $\psi\in \calH_{\mathrm{pp}}(H)$,
    \begin{align*}
        \lim_{T\rightarrow \infty} \frac{1}{T}\vec{Q}_H(T)\psi=0.
    \end{align*}
\end{theorem}
Note that this differs from Simon's classic theorem on the absence of ballistic motion for pure point states \cite{simon1990absence} because it does not assume that $H$ has only pure point spectrum.

We now conclude:
\begin{proof}[Proof of Theorem~\ref{thm:Z2ballistic}]
    The proof follows immediately from combining Proposition~\ref{YTransportThm} and Proposition~\ref{XAbsenceThm}.
\end{proof}

\section{Ballistic transport for the scattering states}\label{scatteringSection}
In this section, we will show that there is a dense set of states $\calD\subset \ell^2(\bbZ^2)$ such that for $\psi \in \calD$, $\Omega \psi $ exhibits ballistic transport. This is based on a criterion for a scattering state to exhibit ballistic transport, which we prove in Appendix~\ref{scatteringAppendix}. In fact, the periodicity plays no role in this section; only the compact support in the $x$-directions matters, as one expects from such scattering arguments. We also comment that the proof idea translates immediately to higher-dimensional and continuum settings.  \par
For $\psi\in \ell^2(\bbZ)$ let its Fourier transform $\hat{\psi}(\xi)\in L^2(\bbT)$ be given by
\begin{align*}
    \hat{\psi}(\xi)=(2\pi)^{-1}\sum_{x\in\bbZ}\psi(x)e^{-ix\xi},
\end{align*}
where $\bbT$ is the torus identified with $[0,2\pi)$ modulo equivalence. Note that this differs from the Floquet transform since it does not include the period $L$.
We define
\begin{align*}
     &\calD_a=\text{Span}(\{\psi_x\otimes \psi_y \in \ell^2(\bbZ)\otimes \ell^2(\bbZ) \mid \hat{\psi}_x\otimes\hat{\psi}_y\in C^\infty(\bbT)\otimes C^\infty(\bbT)\text{ and }\supp \hat{\psi}_x\Subset \tilde{B}_a^c\}),
\end{align*}
where
\begin{align*}
    \tilde{B}_a^c=\{\xi\in\bbT\mid \abs{\sin(\xi)}>a\}.
\end{align*}
Now let $\calD=\bigcup\limits_{a>0}\calD_a$, which is dense in $\ell^2(\bbZ^2)$ because the sets $\tilde{B}_a^c$, for $a>0$, exhaust $\bbT^2$.\par
We begin with the following one-dimensional propagation estimate:

\begin{proposition}\label{DiscNonStat}\label{propEst2}
    Suppose that $\psi \in\ell^2(\bbZ)$ is such that $\hat{\psi}\in C^\infty(\bbT)$  and satisfies
    $\supp \hat{\psi}\subset \tilde{B}_a^c$
    for some $a>0$. Then for any $\ell>0$, there exists a constant $C>0$ depending only on $\psi,\ell,$ and $a$ so that for all $x\in\bbZ$ and $t\in\bbR$ such that $\frac{|x|}{\abs{t}}<a$ we have that
    \begin{align*}
        \abs{(e^{-itH_0}\psi)(x)}< C(1+|x|+\abs{t})^{-\ell},
    \end{align*}
    and also
    \begin{align*}
        \|\chi_{|x|\leq R}e^{-itH_0}\psi\|_{\ell^2(\bbZ)}\leq C(1+|t|)^{-\ell},
    \end{align*}
    uniformly for $t$ and $R$ satisfying $\frac{R}{|t|}<a$.
\end{proposition}
\begin{proof}
We integrate by parts $\ell$ times in the inversion formula
    \begin{align*}
        (e^{-itH_0}\psi)(x)=\int\limits_{\bbT}e^{i(\xi x-2t\cos(\xi))}\hat{\psi}(\xi)\:d\xi
    \end{align*}
    to find that
    \begin{align*}
        (e^{-itH_0}\psi)(x)=\int\limits_\bbT e^{i(\xi x-2t\cos(\xi))}L^\ell(\hat{\psi})\:dk,
    \end{align*}
    where $L$ is the differential operator $f(\xi)\mapsto\frac{d}{d\xi}[\frac{i}{x+2t\sin(\xi)}f(\xi)]$. By the product rule, it follows that
    \begin{align*}
        \left|(e^{-itH_0}\psi)(x)\right|\leq C \max_{\xi\in\supp \hat{\psi}}|x+2t\sin(\xi)|^{-\ell}
    \end{align*}
    for $C$ depending on the first $\ell$ derivatives of $\hat{\psi}$. Now, on the support of $\hat{\psi}$ we have that
    \begin{align*}
        |x+2t\sin(\xi)|>2|t||\sin(\xi)|-|x|>2|t|a-|x|
    \end{align*}
    and it is easy to check that there is some $c>0$ such that
    \begin{align*}
        2|t|a-|x|>c(|x|+|t|)
    \end{align*}
    uniformly for $x$ and $t$ satisfying $\frac{|x|}{|t|}<a$. The first inequality now follows immediately, whereas the second now follows from summing the first inequality.
\end{proof}
We can now prove:
\begin{proposition}\label{scatteringProp}
    Suppose that $V$ is strip periodic on $\ell^2(\bbZ^2)$. Then for any $\psi\in\calD$, we have that $\Omega\psi$ exists and exhibits ballistic transport.
\end{proposition}
\begin{proof}
    By Proposition \ref{BTForFree} it is enough to show that for any $\psi\in \calD_a$, $\ell>0$, and $t$ large enough, there is $C>0$ so that
    \begin{align}
        &\|Ve^{-itH_0}\psi \| \leq C(1+t)^{-\ell}\label{Eq:1FreeBt}\\
        &\|Ve^{-itH_0}\vec{P}\psi\|\leq C(1+t)^{-\ell}\label{Eq:1.5FreeBt}\\
        &\|XVe^{-itH_0}\psi \| \leq C(1+t)^{-\ell}\label{Eq:2FreeBt}\\
        &\|YVe^{-itH_0}\psi \| \leq C(1+t)^{-\ell}\label{Eq:3FreeBt}.
    \end{align}
    where $\vec{P}=(P^x,P^y)$. 
    By linearity it is enough to consider $\psi=\psi_x\otimes \psi_y$.
    For such a state, we may write
    \begin{align*}
        \|Ve^{-itH_0}\psi \|&\leq  \|V\|_{\ell^\infty}\|\chi_{|x|\leq R}e^{-itH_0}\psi\|\\
        &\leq C\|\chi_{|x|\leq R}e^{-itH_0^x}\psi_x\|_{\ell^2(\bbZ)}\cdot\|e^{-itH_0^y}\psi_y\|_{\ell^2(\bbZ)},
    \end{align*}
    and similarly for $\vec{P} \psi$. As the components of $\vec{P} \psi$ are again in $\calD_a$, from Proposition \ref{DiscNonStat} we obtain the inequalities \eqref{Eq:1FreeBt} and \eqref{Eq:1.5FreeBt} for $t>\frac{R}{a}$.\par
    Next, we note that
    \begin{align*}
        \|XVe^{-itH_0}\psi \|\leq R\|V\|_{L^\infty}\|\chi_{|x|\leq R}e^{-it H_0}\psi\|,
    \end{align*}
    so the same argument yields the inequality \eqref{Eq:2FreeBt}, as well. Finally, we have that
    \begin{align*}
        \|YVe^{-itH_0}\psi\|\leq\|V\|_{\ell^\infty}\|\chi_{|x|\leq R}e^{-itH_0^x}\psi_x\|_{\ell^2(\bbZ)}\cdot\|Ye^{-itH_0^y}\psi_y\|_{\ell^2(\bbZ)}.
    \end{align*}
    The last term in the product is bounded by $C(1+t)$ due to Lemma \ref{radinSimonLem}, so again applying Proposition \ref{DiscNonStat} yields \eqref{Eq:3FreeBt}, thus completing the proof.
\end{proof}

\section{Density of $D(\vec{Q})$ in $\calH_{\mathrm{sur}}$}\label{Domains}
In this section, we show that $D(\vec{Q})$ is dense in $\calH_{\mathrm{sur}}$ in order to complete the proof of Theorem~\ref{thm:Z2ballistic}.

First, we give a condition for being in $D(Y)$, using the following Paley-Wiener type theorem (see also Theorem 4.2 in \cite{Kuchment}):
\begin{proposition}\label{paleywiener}
    If $(U\varphi)(k,x,y)\in C^1(\calB;\ell^2(\bbZ\times\bbZ_L))$ then $\varphi\in D(Y)$.
\end{proposition}
\begin{proof}
    Integrating by parts in the inversion formula
     \begin{align*}
        \varphi(x,y) =\int\limits _{\calB}e^{ik y}U\varphi(k,x,y)\:  \frac{dk}{\abs{\calB}},
     \end{align*}
    we have
    \begin{align*}
        y\varphi(x,y)=i\int\limits_\calB e^{iky}\frac{\partial}{\partial k}(U\varphi)(k,x,y)\:\frac{dk}{|\calB|}.
    \end{align*}
    The result now follows from the unitarity of $U$.
\end{proof}
Membership in $D(X)$ is somewhat more delicate. Recall from Proposition~\ref{T0Summary} that if $E_0$ is an eigenvalue of $H(k_0)$ then the set of $j\in\bbZ_L$ such that $e_j(E_0,k_0)\not\in[-2,2]$ must be non-empty. Now we prove:
\begin{proposition}\label{eigenfunctionD(X)}
    Let $\varphi$ be an eigenfunction of $H(k_0)$ with eigenvalue $E_0$ and let $\calJ\subset \bbZ_L$ be the set of modes such that $e_j(E_0,k_0)\not\in[-2,2]$.
    Let
    \begin{align*}
        \delta(E_0,k_0) := \min_{\pm }\min_{j\in \calJ} |\pm 2 - e_j(E_0,k_0)|.
    \end{align*}
    Then there exists $C_\delta$ depending continuously on $\delta$ so that
    \begin{align*}
        \|X \varphi\|_{\ell^2(\bbZ^2)}\leq C_\delta \|\varphi\|_{\ell^2(\bbZ^2)}.
    \end{align*}
\end{proposition}

\begin{proof}
    From Proposition~\ref{T0Summary},
    $\vec{\varphi}_{R-1}\oplus \vec{\varphi}_R$ lies in $\calV^+(E_0,k_0)\subset \bbC^L\oplus \bbC^L$.
    By definition, this subspace is the sum of the eigenspaces of $T_0(E_0,k_0)$ with eigenvalues $\mu^+(e_j(E_0,k_0))$, all of which have modulus less than $1$, for $j\in \calJ$.
    By the continuity of $\mu^+$ from $\bbC\setminus [-2,2]$ to  $\bbD$, we may find some constant $c_\delta<1$ depending only on $\delta$ such that
    $\max_{j\in \calJ}|\mu^+(e_j(E_0,k_0)|<c_\delta$.
    Therefore, because
    \begin{align*}
        \begin{pmatrix}
            \vec{\varphi}_{R+n} \\\vec{\varphi}_{R_n-1}
        \end{pmatrix}
        =T^n(E_0,k_0)\begin{pmatrix}
            \vec{\varphi}_R\\ \vec{\varphi}_{R-1}
            \end{pmatrix}.
    \end{align*}
    we have that
    \begin{align*}
    |\varphi(R+n,y)| \leq c_\delta^n \|\varphi\|_{\ell^2(\bbZ\times \bbZ_L)}.
    \end{align*}
    for all $n>0$.
    The same argument shows that $|\varphi(-R-n,y)| \leq c_\delta^n \|\varphi\|_{\ell^2(\bbZ\times \bbZ_L)}$, so by summing we may conclude.
\end{proof}
Now we can establish the density of $D(\vec{Q})$ in $\calH_{\textrm{sur}}$.
\begin{proposition}\label{pr:D(Q)Density}
    We have that $D(\vec{Q})\cap \calH_{\mathrm{sur}}$ is dense in $\calH_{\mathrm{sur}}$.
\end{proposition}
\begin{proof}
    From Theorem~\ref{thm:analyticity}, any state in $\psi\in \calH_{\mathrm{sur}}$ is given by
    \begin{align*}
        (\calU\psi)(k) = \sum_{l,i} a_{l,i}(k)\varphi_{l,i}(k),
    \end{align*}
    for some functions $a_{l,i}:U_l\rightarrow \bbC$, and each  $\varphi_{l,i}(k,x,y)$ is a normalized eigenfunction of $H(k)$ that is analytic up to a finite set of points.
    In particular, we have that
    \begin{align*}
        \|\psi\|_{\ell^2(\bbZ^2)}^2=\sum_{l,i}\|a_{l,i}\|_{L^2(\calB)}^2.
    \end{align*}
    It follows easily then that the set of  $\psi\in \calH_{\textrm{sur}}$ with $a_{l,i}$ that are $C^1$ and supported away from the points of non-analyticity of $\varphi_{l,i}$ are dense in $\calH_{\textrm{sur}}$.
    By Proposition~\ref{paleywiener}, these states are in $D(Y)$ so we have shown that $D(Y)$ is dense in $\calH_{\textrm{sur}}$.
    
    Now, we show that $D(X)$ is dense in $\calH_{\textrm{sur}}$. For this, let $\lambda_{l,i}(k):U_l\rightarrow \bbR$ be the eigenvalue corresponding to $\varphi_{l,i}(k)$.
    Each $\lambda_{l,i}$ is an analytic function of $k$ except at finitely many points, and thus the set of $k$ for which $ e_j(\lambda_{l,i}(k),k)=\pm2$ for each $j$ is either finite or a full interval.
    In the first case, we observe that the quantity $\delta(\lambda_{l,j}(k),k)$ defined in Proposition~\ref{eigenfunctionD(X)} will be uniformly bounded away from $0$ for $k$ away from a finite set of points.
    On the other hand, if $ e_j(\lambda_{l,i}(k),k)=\pm2$ for all $k$ in an interval then $\delta(\lambda_{l,i}(k),k)$ is again bounded away from $0$ so long as $e_{j'}(\lambda_{l,i}(k),k)\pm 2$ is non-zero, for any $j'\neq j$, which may happen at only finitely many points of $k$.
    In either case, by choosing $a_{l,i}$ to vanish in a neighborhood of these discrete sets, we will obtain an element of $D(X)$ by Proposition~\ref{eigenfunctionD(X)}, from which density is immediate.
    
    Finally, it is easy to see from combining the above constructions that in fact $D(\vec{Q})=D(X)\cap D(Y)$ is dense in $\calH_{\textrm{sur}}$, so we are done.
\end{proof}
Combining this result with Proposition \ref{scatteringProp} yields Theorem \ref{mainCor}:
\begin{proof}[Proof of Theorem \ref{mainCor}]
    Recall from Proposition~\ref{Completeness} that we have
    \begin{align*}
        \ell^2(\bbZ^2)=\calH_\textrm{sur}\oplus \Ran\Omega.
    \end{align*}
    We have just shown in Proposition~\ref{pr:D(Q)Density} that $D(\vec{Q})$ is dense in $\calH_\textrm{sur}$ and we know from Theorem~\ref{thm:Z2ballistic} that this entire set must exhibit ballistic transport.
    Thus, we need only show that a dense subset of $\Ran\Omega$ exhibits ballistic transport. The set $\calD$ of Proposition \ref{scatteringProp} is dense and any $\psi\in \Omega\calD$ exhibits ballistic transport. Since $\Omega$ is a partial isometry we may conclude.
\end{proof}

\appendix

\section{Ballistic transport via Cook's method}\label{scatteringAppendix}
In this section, we give a criterion, similar to the criterion given in Cook's method \cite{RSVol3}, to show certain states exhibit ballistic transport. This lemma gives rigor to the idea that asymptotically free states should exhibit ballistic transport. Though this result seems natural, it appears to be absent from the literature. In \cite{enss1983asymptotic}, the author proves a similar result only in the strong resolvent sense, which does not imply the strong limit sense that we prove here- see \cite{damanik2024ballistic} for the differences.  \par
In this section, we will work in more generality than the rest of the paper: we study operators of the form
\begin{align*}
    H=H_0+V,
\end{align*}
acting either on $\ell^2(\bbZ^d)$ or $L^2(\bbR^d)$, in which case $H_0=-\frac{1}{2}\Delta$, for any $d\geq 1$. 
We assume only that $V$ is real-valued and bounded, though, as is typical in scattering theory, one could likely relax this assumption, for instance, to relative $H_0$-boundedness.\par
We recall that we say that a state $\psi \in \calH\cap D(Q_j)$ (where $Q_j$ is the $j$th component of the position operator) exhibits ballistic transport in the direction of $\e_j$ if we have that 
\begin{align*}
    \lim\limits_{t\rightarrow\infty}\frac{Q_j(t)}{t} \psi
\end{align*}
exists and is nonzero. We also define the momentum operators, $P_j$ which are given by $-i\partial_{x_j}$ on $\bbR^d$ and 
\begin{align*}
    (P_j\psi)(n)=-i(\psi(n+\e_j)-\psi(n-\e_j))
\end{align*}
on $\bbZ^d$.
\par
As above, the wave operator $\Omega$ (when it exists) is defined via the strong limit
\begin{align*}
    \Omega=\slim\limits_{t\rightarrow \infty } \Omega(t)
\end{align*}
for
\begin{align*}
    &\Omega(t)=e^{itH}e^{-itH_0}.
\end{align*}
We also recall that $\psi \in \text{Ran}(\Omega)$ is called asymptotically free, and its evolution is close to the free evolution. Typically, one shows the existence of the wave operator on $\Omega \psi$ via Cook's method \cite[Theorem XI.4]{RSVol3}, which involves controlling the following function of $t$
\begin{align}\label{Cook'sQuantity}
    Ve^{-itH_0}\psi
\end{align}
in $L^2$. Our result is an extension of this method to settings in which one can control \eqref{Cook'sQuantity} in a stronger norm.
\par
Following Radin-Simon \cite{RadinSimon}, we define the following subspaces
\begin{align*}
    S_j(\bbR^d)=\{f\in L^2(\bbR^d)\mid Q_jf\in L^2, P_jf\in L^2\}
\end{align*}
and 
\begin{align*}
    S_j(\bbZ^d)=\{f\in \ell^2(\bbZ^d)\mid Q_jf\in \ell^2(\bbZ^d)\}
\end{align*}
equipped with the following norms
\begin{align*}
    \|f\|_{S_j(\bbR^d)} =\sqrt{\|f\|_{H^1}^2+\|Q_jf\|_2^2}
\end{align*}
and
\begin{align*}
    \|f\|_{S_j(\bbZ^d)}= \sqrt{\|f\|_2^2+\|Q_jf\|_2^2}.
\end{align*}
We emphasize that whenever there is no subscript to the norm, it is simply the $L^2$ (or $\ell^2$) norm. Before proving these facts, we record a simple technical lemma that will be used in the proofs that follow.
\begin{lemma}
    Let $\psi:\bbR\to L^2(\bbR^d)$ or $\psi:\bbR\to \ell^2(\bbZ^d)$ be continuous. Then, for $s<t$,
    \begin{align}\label{eq:interchange}
        \left\|Q_j\int\limits_{s}^t\psi(\tau)\:d\tau\right\|\leq \int\limits_{s}^t\|Q_j\psi(\tau)\|\:d\tau,
    \end{align}
    where either side of the above inequality may be infinite. 
\end{lemma}
\begin{proof}
    We define
    \begin{align*}
        F_N(q_j)=|q_j|\chi_{W_N}(q_j)
    \end{align*}
    for $W_N=\{x\in \bbR^d: |x_j|\leq N\}$, or the discrete analog, and denote by $F_N(Q_j)$ the corresponding multiplication operator.  
    By the monotone convergence theorem, we have
    \begin{align*}
        \|Q_j\varphi\|=\lim_{N\to\infty}\| F_N(Q_j)\varphi\| 
    \end{align*}
    for any $\varphi\in L^2(\bbR^d)$, or $\varphi\in \ell^2(\bbZ^d)$. Thus, 
    \begin{align*}
        \left\|Q_j\int\limits_{s}^t\psi(\tau)\:d\tau\right\|& =\lim_{N\to\infty}\left\|F_N(Q_j)\int\limits_{s}^t\psi(\tau)\:d\tau \right\|=\lim_{N\to\infty}\left\|\int\limits_{s}^tF_N(Q_j)\psi(\tau)\:d\tau \right\|\\
        &\leq \lim_{N\to\infty}\int\limits_{s}^t\|F_N(Q_j)\psi(\tau)\|\:d\tau
    \end{align*}
    using the boundedness of $F_N(Q_j)$. Using the monotone convergence theorem again, \eqref{eq:interchange} follows.
\end{proof}
We will also need this ballistic upper bound:
\begin{lemma}\label{radinSimonLem}
    Suppose that $V\in L^\infty(\bbR^d)$ and $\psi\in D(Q_j)$ for some index $j$. Then $e^{-itH}\psi\in D(Q_j)$ for all $t\in\bbR$ and there is some $C>0$ so that
    \begin{align*}
        \|Q_je^{itH}\|_{S_j(\bbR^d)\rightarrow L^2(\bbR^d)}<C(1+|t|).
    \end{align*}
    Similarly, suppose that $V\in \ell^\infty(\bbZ^d)$ and $\psi\in D(Q_j)$ for some index $j$. Then $e^{-itH}\psi\in D(Q_j)$ for all $t\in\bbR$ and there is some $C>0$ so that
    \begin{align*}
        \|Q_je^{itH}\|_{S_j(\bbZ^d)\rightarrow \ell^2(\bbZ^d)}<C(1+|t|).
    \end{align*}
\end{lemma}
\begin{proof}
    The first claim follows from Theorem 2.1 of \cite{RadinSimon}.
    The second follows immediately from the triangle inequality: 
    \begin{align*}
        \|(Q_j)_H(t)\psi\|\leq \|Q_j\psi\|+\left|\int_0^t\|(P_{j})_H(s)\|\|\psi\| ds\right|\leq \|Q_j\psi\|+2|t|\|\psi\|
    \end{align*}
    yielding the bound with $C=2\sqrt{2}$.

\end{proof}
With these results in hand, we can prove the following:
\begin{lemma}
    \label{BTForFree}
    Let $V\in C^1(\bbR^d)\cap L^\infty(\bbR^d) $ and $\nabla V\in L^\infty(\bbR^d)$. Suppose that for some index $j$, $\psi\in \calD(Q_j)\cap H^1(\bbR^d)$ satisfy
    \begin{align}\label{S1Assum}
        &\int\limits_0^\infty (1+t)\|Ve^{-itH_0}\psi\|_{S_j(\bbR^d)} \, dt <\infty,
    \end{align}
    and also
    \begin{align}\label{P_2Assum}
        &\int\limits_0^\infty (1+t)\|Ve^{-itH_0}P_j\psi\|_{L_2(\bbR^d)} \, dt <\infty.
    \end{align}
    Then we have that $\Omega \psi $ and $\Omega P_j\psi$ exist and furthermore
    \begin{align}\label{asyVel}
        \lim\limits_{t\rightarrow \infty }\frac{(Q_j)_H(t)}{t}\Omega \psi=2\Omega P_j\psi.
    \end{align}
    In particular, we have that $\Omega\psi$ exhibits ballistic transport in the $Q_j$-direction.\par
    Similarly, for $V\in \ell^\infty(\bbZ^d)$, let $\psi \in\calD(Q_j) $ satisfy
    \begin{align}\label{S1AssumDisc}
        &\int\limits_0^\infty (1+t)\|Ve^{-itH_0}\psi\|_{S_j(\bbZ^d)} \, dt <\infty,
    \end{align}
     and also
    \begin{align}\label{P_2AssumDisc}
        &\int\limits_0^\infty (1+t)\|Ve^{-itH_0}P_j\psi\|_{\ell^2(\bbZ^d)} \, dt <\infty.
    \end{align}
    Then we have that $\Omega \psi $ exists and furthermore
    \begin{align}\label{asyVelDis}
        \lim\limits_{t\rightarrow \infty }\frac{(Q_j)_H(t)}{t}\Omega \psi=\Omega P_j\psi.
    \end{align}
\end{lemma}

\begin{proof}
    First, we note that the assumption \eqref{S1Assum} (or \eqref{S1AssumDisc}) shows that
    \begin{align*}
        \int\limits_0^\infty \|Ve^{-itH_0}\psi\| \, dt <\infty,
    \end{align*}
    which implies the existence of $\Omega \psi $ by Cook's method \cite{RSVol3}, and similarly, assumptions \eqref{P_2Assum} and \eqref{P_2AssumDisc} show the existence of $\Omega P_j\psi$ .\par 
   To prove \eqref{asyVel}, or \eqref{asyVelDis}, we will start by establishing that $\Omega \psi \in D(Q_j)$, by showing that $\{Q_j\Omega(t)\psi\}$ is Cauchy. This suffices because $\Omega(t)\psi\in D(Q_j)$ for all $t$ and $Q_j$ is a closed operator. By Lemma \ref{radinSimonLem} $\Omega (t)\psi \in D(Q_j)$ for any $t>0$, so we can write for $t>s>0$ 
    \begin{align*}
        &\|Q_j(\Omega(t)-\Omega(s))\psi\|=\|Q_j\int\limits_s^te^{i\tau H}(iV)e^{-i\tau H_0}\psi d\tau \|\leq \int\limits_s^t\|Q_je^{i\tau H}Ve^{-i\tau H_0}\psi\|_{L^2}d\tau\\
        &\leq \int\limits_s^t\|Q_je^{i\tau H}\|_{S_j\mapsto L^2/\ell^2}\|Ve^{-i\tau H_0}\psi\|_{S_j}d\tau,
    \end{align*}
    where $S_j$ is either $S_j(\bbR^d)$ or $S_j(\bbZ^d)$, here and in the following.\par
    Again by Lemma \ref{radinSimonLem}, we have
    \begin{align*}
        \|Q_je^{i\tau H}\|_{S_j\mapsto L^2/\ell^2}\leq C(1+\tau),
    \end{align*}
    so that 
    \begin{align*}
        &\|Q_j(\Omega(t)-\Omega(s))\psi\|\leq \int\limits_s^t(1+\tau)\|Ve^{-i\tau H_0}\psi\|_{S_j}\:d\tau.
    \end{align*}
    This is the tail of $\int\limits_0^\infty(1+\tau)\|Ve^{-i\tau H_0}\psi\|_{S_j}d\tau$, which converges by assumption. Thus, the sequence is Cauchy.\par
    Now, we note that the intertwining property $e^{-itH}\Omega=\Omega e^{-itH_0}$ implies that for $t>0$
    \begin{align*}
        &\|Q_je^{-itH}\Omega \psi -Q_je^{-itH_0}\psi\|=\|Q_j\Omega e^{-itH_0} \psi -Q_je^{-itH_0}\psi\|=\|Q_j(\Omega-\Id) e^{-itH_0}\psi\|\\
        &= \left\|Q_j\int\limits_0^\infty e^{isH}iVe^{-isH_0}e^{-itH_0}\psi\, ds\right\|\leq \int\limits_0^\infty \|Q_je^{isH}\|_{S_j\mapsto L^2/\ell^2}|\|Ve^{-i(s+t)H_0}\psi\|_{S_j}ds\\
        &\leq \int\limits_0^\infty (1+s)\|Ve^{-i(s+t)H_0}\psi\|_{S_j}\:ds
    \end{align*}
    as in the above.  
    Then we can write:
    \begin{align*}
         &\|Q_je^{-itH}\Omega \psi -Q_je^{-itH_0}\psi\|\leq \int\limits_0^\infty (1+s)\|Ve^{-i(s+t)H_0}\psi\|_{S_j}\:ds\\
         &= \int\limits_t ^\infty (1+s-t)\|Ve^{-isH_0}\psi\|_{S_j}\:ds \leq \int\limits_t ^\infty (1+s)\|Ve^{-isH_0}\psi\|_{S_j}\:ds.
    \end{align*}
    Thus, we may conclude that
    \begin{align*}
        \lim\limits_{t\rightarrow\infty } \|Q_je^{-itH}\Omega \psi -Q_je^{-itH_0}\psi\|\leq  \lim\limits_{t\rightarrow\infty }\int\limits_t ^\infty (1+s)\|Ve^{-isH_0}\psi\|_{S_j}ds=0
    \end{align*}
    by our assumptions. \par
    This implies that 
    \begin{align*}
        \lim\limits_{t\rightarrow \infty }\|e^{itH}Q_je^{-itH}\Omega \psi-e^{-itH}Q_je^{-itH_0}\psi \|=0,
    \end{align*}
    so in order to establish \eqref{asyVel}, it is enough to show that $\{\frac{1}{t}e^{-itH}Q_je^{-itH_0}\psi\}$ converges. \par 
    In the continuous setting, we use the Fourier transform to see that
    \begin{align*}
        \frac{1}{t}e^{itH}Q_je^{-itH_0}\psi&=\frac1te^{itH}\mathcal{F}^{-1}(2te^{-it|\xi_j|^2}\xi_j\hat{\psi}-e^{-it|\xi_j|^2}i\partial_{\xi_j} \hat{\psi})\\
        &=e^{itH}\mathcal{F}^{-1}(2e^{-it|\xi_j|^2}\xi_j\hat{\psi})+O_{L^2(\bbR^d)}(1/t).
    \end{align*}
    The first term is $2\Omega(t) P_j\psi $, which we established above converges.\par
    In the discrete setting, we have
     \begin{align*}
        \frac{1}{t}e^{itH}Q_je^{-itH_0}\psi&=\frac1te^{itH}\mathcal{F}^{-1}\left(-2\sin(\xi_j)te^{-it\sum\limits_{\ell=1}^d2\cos(\xi_\ell)}\hat{\psi}-e^{-it\sum\limits_{\ell=1}^d2\cos(\xi_\ell)}i\partial_{\xi_j} \hat{\psi}\right)\\
        &=e^{itH}\mathcal{F}^{-1}\left(-2\sin(\xi_j)e^{-it\sum\limits_{\ell=1}^d2\cos(\xi_\ell)}\hat{\psi}\right)+O_{\ell^2(\bbZ^d)}(1/t).
    \end{align*}
    As before, the first term is $\Omega(t)P_j\psi$. Thus, we may conclude that
    \begin{align*}
        \lim_{t\to\infty}\frac1t(Q_j)_H(t)\psi =\lim_{t\to\infty}\frac{1}{t}e^{itH}Q_je^{-itH_0}\psi=\Omega P_j\psi,
    \end{align*}
    which is clearly non-zero because $\|\Omega P_j\psi\|=\|P_j\psi\|$ is positive for $\psi\ne 0$, $\psi\in H^1(\bbR^d)$, or $\psi \in \ell^2(\bbZ^d)$, as needed.
\end{proof}
As a corollary, we obtain a transport result for potentials that decay faster than short-range:
\begin{corollary}
    Let $V\in C^1(\bbR^d)\cap L^\infty (\bbR^d)$ be such that 
    \begin{align*}
        &\|(1+R)\vec{Q}\chi_{|x|>R}V\|_{\infty }\in L^1([0,\infty),dr),&&\|(1+R)\chi_{|x|>R}\nabla V\|_{\infty }\in L^1([0,\infty),dr).
    \end{align*}
    Recall the sets
    \begin{align*}
        &\calD_a=\{\psi \in \calS(\bbR^d)\mid \supp \hat{\psi}\Subset B_a^c\},&&\calD= \bigcup\limits_{a>0}\calD_a.
    \end{align*}
    Then for all $\psi\in\calD$, $\Omega \psi $ exists and exhibits ballistic transport in all directions.
\end{corollary}
\begin{proof}
    This proof is very similar to the proof of Proposition \ref{scatteringProp}, only allowing for short-range decay instead of compact support.\par 
    We note that we can write 
    \begin{align*}
        \|Ve^{-itH_0}\psi \|_{H_1}&\leq C( \|Ve^{-itH_0}\psi \|+ \|\nabla Ve^{-itH_0}\psi \|+\|V e^{-itH_0}\nabla \psi\|)\\
        &\leq C( \|Ve^{-itH_0}\psi \|+ \|\nabla Ve^{-itH_0}\psi \|+\|V e^{-itH_0}P \psi\|).
    \end{align*}
    We start with the first term and consider $0<\varepsilon<a$
    \begin{align*}
         \|Ve^{-itH}\psi\|&\leq \|V\chi_{|x|<\varepsilon t}e^{-itH_0}\psi\|+\|V\chi_{|x|>\varepsilon t}e^{-itH_0}\psi\|\\
        &\leq \|V\|_{\infty}\|\chi_{|x|<\varepsilon t}e^{-itH_0}\psi\|+\|V\chi_{|x|>\varepsilon t}\|_{\infty}\|\psi\|.
    \end{align*}
    By  writing 
    \begin{align*}
        e^{-itH_0}\psi(x)=(2\pi)^{-\frac{d}{2}}\int\limits_{\bbR^d}e^{i(x\cdot\xi-t\xi^2)}\hat{\psi}(\xi)\:d\xi
    \end{align*}
    and the principle of non-stationary phase, namely the Corollary to Theorem XI.14 of \cite{RSVol3}, one is able to get similar statement to Proposition \ref{DiscNonStat} in the continuum, and conclude that  since $\psi \in \calD_a$, for any large enough $\ell>0$, and some $C>0$ we have the bound
    \begin{align*}
        \|\chi_{|x|<\varepsilon t}e^{-itH_0}\psi\|^2\leq \int\limits_{B_{\varepsilon t}}C(1+|x|+|t|)^{-\ell} \, dx\leq C \frac{(\varepsilon t)^d}{(1+|t|)^\ell}<C(1+|t|)^{-\ell+d}.
    \end{align*}
    Thus, we get that 
    \begin{align*}
         \|Ve^{-itH}\psi\|&\leq C(1+|t|)^{-\ell}+\|V\chi_{|x|>\varepsilon t}\|_{\infty}\|\psi\|.
    \end{align*}
    Noting that $P \psi \in \calD_a$ as well, we can get the following
    \begin{align*}
        \|Ve^{-itH_0}\psi \|_{H_1}&\leq C[(1+|t|)^{-\ell}+\|V\chi_{|x|>\varepsilon t}\|_{\infty}\|\psi\|+\|V\chi_{|x|>\varepsilon t}\|_{\infty}\|P\psi\|+\|\nabla V\chi_{|x|>\varepsilon t}\|_{\infty}\|\psi\|],
    \end{align*}
    so in particular we have that 
    \begin{align*}\label{S1Assum}
        &\int\limits_0^\infty (1+t)\|Ve^{-itH_0}\psi\|_{H^1} \, dt\\
        &<C[\int\limits_0^\infty (1+t)(1+|t|)^{-\ell}\, dt+\|\psi\|_{H^1} \int\limits_0^\infty (1+t)\|V\chi_{|x|>\varepsilon t}\|_{\infty}\, dt+\|\psi\|\int\limits_0^\infty (1+t)\|\nabla V\chi_{|x|>\varepsilon t}\|_{\infty}\, dt].
    \end{align*}
    By assumption, we get that the last terms are finite, and the first is finite for $\ell>0$ large enough. This shows that 
    \begin{align*}
        &\int\limits_0^\infty (1+t)\|Ve^{-itH_0}P\psi\| \, dt <\infty
    \end{align*}
    as well. \par
    Finally, it remains to show that
    \begin{align*}
        \int\limits_0^\infty (1+t)\|\vec{Q}Ve^{-itH_0}\psi\| \, dt<\infty. 
    \end{align*}
    We note that for $\psi \in \calD_a$ we can write
    \begin{align*}
        \|\vec{Q}Ve^{-itH}\psi\|&\leq \|\vec{Q}V\chi_{|x|<\varepsilon t}e^{-itH_0}\psi\|+\|\vec{Q}V\chi_{|x|>\varepsilon t}e^{-itH_0}\psi\|\\
        &\leq \varepsilon t\|V\|_{\infty}\|\chi_{|x|<\varepsilon t}e^{-itH_0}\psi\|+\|\vec{Q}V\chi_{|x|>\varepsilon t}\|_{\infty }\|\psi\|.
    \end{align*}
    By the above, we have that for some large enough $\ell>0$, and some $C>0$
    \begin{align*}
        \|\chi_{|x|<\varepsilon t}e^{-itH_0}\psi\|\leq C(1+|t|)^{-\ell},
    \end{align*}
    so we find
    \begin{align*}
        \int\limits_0^\infty (1+t)\|\vec{Q}Ve^{-itH_0}\psi\| \, dt<\int\limits_0^\infty (1+t)C(1+|t|)^{-\ell} \, dt+C\int\limits_0^\infty (1+t)\|\vec{Q}V\chi_{|x|>\varepsilon t}\|_{\infty}\|\psi\| \, dt.
    \end{align*}
    The second term is finite by assumption, and the first converges for $\ell$ large enough.
\end{proof}

\section{A Form of Weierstrass Preparation}\label{WeierstrassPreparation}
Recall that a Weierstrass polynomial on $U\subset \bbR^2$, a neighborhood of $(0,0)$, is a function of the form
\begin{align*}
    F(x,y)=x^n+g_{n-1}(y)x^{n-1}+\cdots+g_0(y)
\end{align*}
where each $g_k$ is analytic and satisfies $g_k(0)=0$.\par
Throughout this section, we will use the following notation. For a function $g$, we denote by $Z_g$ its zero set. If $f$ is a monic polynomial with complex roots $\{\alpha_i\}$, we recall that its discriminant is (up to a sign)
\begin{align*}
    \Delta=\prod\limits _{j\neq k}^n (\alpha_j-\alpha_k).
\end{align*}
The complex analog of the following result may be gleaned from Chapter 1 of \cite{chirka} (see also Chapter 6 of \cite{lebl}). However, we include a proof below for the convenience of the reader.
\begin{lemma}\label{WeierstrassPreparationLemma}
    Let $V$ be an open subset of $ \bbR^2 $ containing $ (0,0) $. Suppose that $f:\bbR^2\to \mathbb R$ is analytic, vanishes at $(0,0)$, and $ z\mapsto f(0,z) $ is not identically $ 0 $.
    Then there exist open intervals $ I $ and $J$ containing $ 0 $ with $I\times J\subset V$ and a Weierstrass polynomial $F$ such that $ Z_F=Z_f $ on $ I\times J $ and the discriminant of $ F $ is not identically $0$ on $ I $.
\end{lemma}
\begin{proof}
	Let $ \tilde{f} $ be the complexification of $ f $, i.e. its extension to a complex analytic function on $ \tilde{V} $ a neighborhood of $ (0,0) $ in $ \bbC^2 $, which contains $ V $. It suffices to find a Weierstrass polynomial $ F $ on some polydisc $ D_1\times D_2 \subset \bbC^2$ centered at $ (0,0) $ with $Z_F=Z_{\tilde{f}}$ with a discriminant that does not vanish identically in $D_1\times D_2$, because by the identity principle for single variable analytic functions, the discriminant cannot vanish on a real interval $ I\subset D_1 $ without vanishing on $D_1\times D_2$.\par
	With this in mind, by the assumption $ f(0,z)\not\equiv 0 $ and the identity principle, we may choose a disc centered at $ 0 $, $ D_2\subset \bbC $, so that $ 0 $ is the only zero of $ f(0,z) $ inside $ D_2 $. By Rouch\'{e}'s theorem and continuity of $f(\cdot,z)$, we may find a disc $ D_1\subset \bbC $ containing $ 0 $ and so that for each $ w\in D_1 $, the function $ f_w(z):D_2\rightarrow \bbC $ given by $ f_w(z)=f(w,z)$  has exactly $ m $ zeroes (counted with multiplicity) in $ D_2'=\frac{1}{2}D_2 $. Now let $ r:=\sup_{w\in D_1}|\{z\in D_2:f_w(z)=0\}|\leq m$ be the maximal number of \emph{geometrically distinct} zeroes of $ f_w(z) $ for all $w\in D_1$. Let $ U $ be the set of $ w\in D_1 $ so that $ z\mapsto f(w,z) $ has $ r $ geometrically distinct zeroes. \par
	For any $ w_0\in U $, let $ \alpha_1(w_0),\ldots,\alpha_r(w_0) $ be the roots of $ f_{w_0}(z) $ labeled arbitrarily. For each $ 1\leq j \leq r $, let $ C_j\subset D_2 $ be a circle containing only $ \alpha_j(w_0) $ so that $ f_{w_0}(z) $ is non-zero on each $ C_j $. By continuity and the compactness of the $ C_j $, we may find a neighborhood of $ w_0 $, $ N $, so that $ f_{w}(z) $ is non-zero on each $ C_j $ for all $ w\in N $. The argument principle shows that for all $w\in N$, the number of zeroes (counted with multiplicity) of $ f_{w}(z) $ in each $ C_j $ is $ r_j $, the multiplicity of $ \alpha_j(w_0)$ as a zero of $ f_{w_0}(z) $. However, by the maximality of $ r $, there can only be exactly one geometrically distinct zero, which is given by the formula
	\begin{align*}
		\frac{1}{2\pi i r_j}\int\limits_{C_j}z\frac{f_{w}'(z)}{f_{w}(z)}\; dz
	\end{align*}
	as a standard consequence of the residue theorem. Since this expression is clearly analytic in $ w $ (for instance, by differentiating with $ \frac{\partial }{\partial \bar{w}}$ under the integral), we see that each $ \alpha_j $ extends to an analytic function in some neighborhood of $ w $, and in particular that $ U $ is open.\par
	Now, let $ \Delta(w):U\rightarrow\bbC $ be given by
	\begin{align*}
	\Delta(w)=\Pi_{j\neq k }(\alpha_j(w)-\alpha_k(w)),
	\end{align*}
	which is well-defined and analytic because the expression is independent of the labeling of the $ \alpha_j $. We wish to show that by defining $ \Delta(w) $ to be $ 0 $ on $ D_1\setminus U $ we obtain a continuous function or equivalently that if $ w_n\rightarrow w\in D_1\setminus U $ then $ \Delta (w_n)\rightarrow 0 $. Thus, as roots approach the boundary of $ U $ in $D_1$, they must join, as the number of roots is smaller outside of $ U $. To be precise, for each $ n $ let $ \alpha_1(w_n),\ldots, \alpha_r(w_n) $ be the $ r $ distinct roots of $ f_{w_n}(z) $ and observe that the sequence $ \alpha_1(w_n) $ lies in the compact set $ \overline{D_2'} $ and therefore a subsequence $ \alpha_1(w_{n_k}) $ converges to some $ \alpha_1(w) $. Applying this argument iteratively $r$-times, beginning with the sequence $ w_{n_k} $, we find a subsequence of $ \{w_n\}_{n=1}^\infty $, call it $ \{w_{n_\ell}\}_{\ell=1}^\infty $,  so that $ \alpha_j(w_{n_\ell})\rightarrow \alpha_j(w) $. By the continuity of $ f(z,w) $, each $ \alpha_j(w) $ is a root of $ f_w $ but because $ w\not\in U $ we must have that $ \alpha_j(w)=\alpha_k(w) $ for some $ j\neq k $. Examining the definition of $ \Delta(w) $, we have shown that every subsequence $ w_n\rightarrow w\in D_1\setminus U $ has a further subsequence along which $ \Delta (w_{n_\ell})\rightarrow 0 $, which establishes the desired continuity.\par
	We have seen that we may extend $ \Delta $ to all of $ D_1 $ as a function which is continuous on all of $ D_1 $ and analytic on $ U $. It now follows from Rad\'{o}'s theorem \cite[Section A1.5]{chirka} that $ \Delta $ is in fact analytic on $ D_1 $ so that its zero set (i.e. $ D_1\setminus U $) consists of isolated points. Now observe that on $ U\times D_2 $ 
	\begin{align*}
		F(w,z)=(z-\alpha_1(w))\cdots (z-\alpha_r(w))
	\end{align*}
for each $ z $ is well-defined and analytic because it is independent of the labeling of the roots. Moreover, it is bounded because each $ \alpha_j(w) $ lies in $ D_2' $ so that for each $ z $ we may extend it to an analytic function on $ D_1 $ by Riemann's removable singularity theorem. Thus, $ F $ is a Weierstrass polynomial whose roots coincide with those of $ f $ on $ U $ and therefore also on all of $ D_1 $ by the identity principle. By construction, its roots are simple outside of $ Z_\Delta $ so its discriminant is non-vanishing there.
\end{proof}

\bibliographystyle{amsplain}
\bibliography{bibliography}
\end{document}